\documentclass[11pt]{article}

\usepackage{amssymb}
\usepackage{amsthm}
\usepackage{amsmath}

\usepackage{graphicx} 
 \usepackage{tikz}               

\usepackage{subcaption}
\usepackage{array}

\usepackage[hidelinks]{hyperref}

\usepackage[linesnumbered,ruled,vlined]{algorithm2e}

\numberwithin{equation}{section}

\newtheorem{theorem}{Theorem}
\newtheorem{proposition}[theorem]{Proposition}
\newtheorem{corollary}[theorem]{Corollary}
\newtheorem{lemma}[theorem]{Lemma}

\theoremstyle{definition}
\newtheorem{remark}[theorem]{Remark}

\def\beq{\begin{equation}}
\def\beql#1{\beq\label{#1}}
\def\eeq{\end{equation}}
\def\beqa{\beq\begin{aligned}}
\def\beqal#1{\beql{#1}\begin{aligned}}
\def\eeqa{\end{aligned}\eeq}
\def\bseq{\begin{subequations}}
\def\bseql#1{\begin{subequations}\label{#1}}
\def\eseq{\end{subequations}}

\def\deq#1{(\ref{#1})}

\def\rb#1{\left(#1\right)}
\def\frb#1{\!\rb{#1}}
\def\brb#1{\big(#1\big)}
\def\Brb#1{\Big(#1\Big)}

\def\curb#1{\left\{#1\right\}}
\def\abs#1{\left|#1\right|}

\def\angb#1{\left\langle #1 \right\rangle}

\def\veps{\varepsilon}

\def\G{\Gamma}
\def\g{\gamma}
\def\lam{\lambda}

\def\Om{\Omega}

\def\del{\delta}

\def\vphi{\varphi}

\def\wt{\widetilde}
\def\ol{\overline}

\def\R{\mathbb{R}}

\def\cB{\mathcal{B}}

\def\cH{\mathcal{H}}

\def\cJ{\mathcal{J}}

\def\cL{\mathcal{L}}

\def\cT{\mathcal{T}}

\def\cV{\mathcal{V}}

\def\sbt{\subset}

\def\p{\partial}

\def\what{\widehat}

\def\dvg{\operatorname{\nabla\cdot}}

\def\Span{\mathop{\mathrm{span}}}
\def\dim{\operatorname{dim}}

\graphicspath{{./figures/}}

\newcommand{\asserRef}[1]{\noindent\ref{#1}.\ \ }

\newcommand{\dist}{\operatorname{dist}}
\newcommand{\argmin}{\operatorname*{argmin}}

\newcommand{\uTotal}{u}            
\newcommand{\uTotalMin}{\uTotal_*}            
\newcommand{\uMedium}{\widetilde{u}}                
\newcommand{\uBG}{u^0}                  
\newcommand{\y}{y}                      
\newcommand{\yObs}{y^{\text{obs}}}      
\newcommand{\supp}[1]{\operatorname{supp}\!\left(#1\right)}
\newcommand{\PsiDim}{J}

\begin{document}


\title{Adaptive Spectral Decompositions \\ For Inverse Medium Problems}
%
\author{%
		Daniel H.\ Baffet\,%
			\footnotemark[1]\, \footnotemark[2]\qquad
		Marcus J.\ Grote\,%
			\footnotemark[1]\, \footnotemark[3]\qquad
		Jet Hoe Tang\,%
			\footnotemark[1]\, \footnotemark[4]
		}


\renewcommand{\thefootnote}{\fnsymbol{footnote}}
\footnotetext[1]{Department of Mathematics and Computer Science,
		University of Basel, Basel, Switzerland}
\footnotetext[2]{daniel.baffet@unibas.ch}
\footnotetext[3]{marcus.grote@unibas.ch}
\footnotetext[4]{jet-hoe.tang@univ-grenoble-alpes.fr}

\maketitle
\markboth{}
	{Adaptive Spectral Decompositions For Inverse Medium Problems}

\begin{abstract}
Inverse medium problems involve the reconstruction of a spatially varying unknown medium from available observations by exploring a restricted search space of possible solutions.
Standard grid-based representations are very general but all too often computationally prohibitive due to the high dimension of the search space.
{\it Adaptive spectral (AS) decompositions} instead expand the unknown medium in a basis of eigenfunctions of a judicious elliptic operator, which depends itself on the medium.
Here the AS decomposition is combined with a standard inexact Newton-type method for the solution of time-harmonic scattering problems governed by the Helmholtz equation.
By repeatedly adapting both the eigenfunction basis and its dimension, the resulting adaptive spectral inversion (ASI) method substantially reduces the dimension of the search space during the nonlinear optimization.
Rigorous estimates of the AS decomposition are proved for a general piecewise constant medium.
Numerical results illustrate the accuracy and efficiency of the ASI method for time-harmonic inverse scattering problems, including a salt dome model from geophysics.

\end{abstract}



%
%


%
%
%
%
%
%
%
%
%
%
%
%
%
%
%
%
\section{Introduction}
Inverse medium problems occur in a wide range of applications such as
medical imaging, geophysical exploration and non-destructive
testing. Given observations of a physical state variable, $y$, from the boundary
of a bounded region $\Omega$, one seeks to reconstruct (unknown)
spatially varying medium properties, $u(x)$, inside $\Omega$. In inverse scattering
problems, for instance, $u$ characterizes the location, shape or physical properties of the scatterer, typically a collection of bounded or penetrable inclusions, 
while the scattered wave field $y$ satisfies the governing  (time-dependent or time-harmonic) wave equation. In seismic imaging, in particular, 
full-waveform inversion reconstructs high-resolution subsurface models of the medium parameters $u$ (e.g. spatially varying sound speed) from reflected seismic waves at the Earth's surface  \cite{OV2009}.
To determine $u(x)$ from boundary measurements of $y$, the
inverse medium problem is usually reformulated as a PDE-constrained optimization
problem for a cost functional $\cJ[u]$, which measures the misfit
between the simulated and the observed data \cite{Tarantola:1943:ISR,Haber:2000:OOT}.

Two difficulties one must address when solving an inverse medium problem numerically are the well-posedness of the problem and the size of the set of candidate functions for $u$.
It is well-known that the problem of minimizing the misfit over $L^\infty(\Om)$ is, in general, ill-posed.
To obtain a well-posed problem, the misfit functional is typically modified by adding a Tikhonov-type regularization term \cite{Tikhonov:1943:SIP,EHN2000}.
Then the resulting formulation may be discretized and solved numerically \cite{Metivier:2013:FWI,Grote:2014:IIP}. 
Although Tikhonov regularization generally improves the stability of the inversion,
it does not address the cost of solving an optimization problem in a subspace of high dimension, still determined by the number of degrees of freedom of the spatial numerical discretization. 

When a parametrization of the unknown medium $u(x)$ is explicitly known a priori, 
the inversion can easily be limited to a reduced set of unknown parameters. Moreover, 
if the search space is of sufficiently low dimension, the discretization itself can have a regularizing effect \cite{Chavent,KO2012}. In general, however, such a low-dimensional representation
is not explicitly known a priori. 

To resolve the vexing dilemma of remaining sufficiently general while keeping the dimension of the search space sufficiently small, various sparsity promoting strategies were proposed in recent years. 
In \cite{DDM2004}, Daubechies, Defrise and De Mol considered linear inverse problems yet 
replaced the usual quadratic regularizing penalty by weighted $\ell_p$-penalties on the coefficients 
to promote a sparse expansion of $u$ with respect to an orthonormal basis. 
Later, Loris et al. \cite{LDNDR2010} successfully applied $\ell_1$-norm
soft-thresholding to promote a sparse wavelet representation of $u$ for seismic tomography. 
Similarly, a truncated wavelet representation of the acoustic velocity and 
mass density was used for seismic FWI in \cite{LAH2012}. 
In \cite{HH2008}, a curvelet-based representation of the wave field $y$ was used 
for seismic data recovery from a regularly sampled grid with traces missing.
Recently, Kadu, van Leeuwen and Mulder \cite{KLM2017}
combined a level-set representation, constructed from radial basis functions, with a Gauss-Newton approximation to regularize FWI in the presence of salt bodies.

%

In \cite{BO2010}, de Buhan and Osses proposed to restrict the search space to the span 
of a small basis of eigenfunctions of a judicious elliptic operator,
 repeatedly adapted during the nonlinear iteration. Their approach relies on a decomposition
\begin{equation}\label{eq:wExpansion}
	w=\vphi_0+\sum_{k=1}^\infty \beta_k\vphi_k ,
\end{equation}
for functions $w\in W^{1,\infty}(\Om)$.
Here $\vphi_0$ satisfies the elliptic boundary-value problem
\begin{equation}\label{eq:phi0BVP}
	L_\veps[w]\vphi_0 =0
	\quad \text{in $\Omega$,}
	\qquad
        \vphi_0 = w
	\quad \text{on $\partial\Omega$,}
\end{equation}
and for $k\ge 1$, each $\vphi_k$ is an eigenfunction of a $w$-dependent, linear, symmetric, and elliptic operator $L_\veps[w]$, that is, $\vphi_k$ satisfies
\begin{equation}\label{eq:eigenValProb}
	L_\veps[w]\vphi_k =\lam_k \vphi_k
	\quad \text{in $\Omega$,}
	\qquad
	\vphi_k =0
        \quad \text{on $\partial\Omega$,}
\end{equation}
for an eigenvalue $\lam_k\in\R$. The eigenvalues  $(\lam_k)_{k\ge 1}$ form a nondecreasing sequence, with each eigenvalue repeated according to its multiplicity,
and the set $(\vphi_k)_{k\ge 1}$ is an orthonormal basis of $L^2(\Om)$.
%
Henceforth we shall refer to \deq{eq:wExpansion} as the \emph{adaptive spectral (AS) expansion or decomposition}.

The AS decomposition has been used in various iterative Newton-like inversion algorithms \cite{BK2013,GKN2017,GN2019} as follows:
Given an approximation of the medium, $u^{(m-1)}$, from the previous iteration, the approximation
 $u^{(m)}$ at the current iteration is set as the minimizer of the misfit functional $\cJ[u]$ in the affine space $\vphi_0 + \Span(\vphi_k)_{k=1}^K$, where $\vphi_0$ and $\vphi_k$, $k=1,\ldots,K$, satisfy \deq{eq:phi0BVP} and \deq{eq:eigenValProb} with $w=u^{(m-1)}$, respectively. As the approximation $u^{(m)}$ changes from
 one iteration to the next, so does the affine search space used for the subsequent minimization.

Clearly, the choice of  $L_\veps[w]$ is crucial in obtaining an efficient approximation of the medium
with as few basis functions as possible. The particular choice for $L_\veps[w]$ used in \cite{BO2010}, which
essentially coincides with the linearization of the gradient of the penalized total variation (TV) functional~\cite{GKN2017}, yields a remarkably efficient approximation for piecewise constant functions.
Although $L_\veps[\uTotal]$ is not well-defined for piecewise constant $\uTotal$, numerical discretizations
of  \deq{eq:phi0BVP} or \deq{eq:eigenValProb} can be viewed as approximations with $w$ replaced by a more regular approximation of $\uTotal$, such as a projection of $\uTotal$ on an $H^1$-conforming finite element (FE) space. The AS decomposition also bears a striking similarity to 
nonlinear eigenproblems for the (subdifferential of the) TV functional and, more generally, for one homogeneous functionals used in image processing -- see \cite{BGM2016,BCEGM2016,BCN2002} and the references therein. Depending on a priori available information about the smoothness or spatial anisotropy of the medium, different choices for $L_\veps[w]$ will
yield more or less efficient AS representations \cite{GN2019}.

By combining the adaptive inversion process with the TRAC (time reversed absorbing condition) approach, de Buhan and Kray~\cite{BK2013} developed an effective solution strategy for time-dependent inverse scattering problems.
In~\cite{GKN2017}, Grote, Kray and Nahum proposed the AEI (adaptive eigenspace inversion) algorithm for inverse scattering problems in the frequency domain. It combines the AS decomposition 
with frequency stepping and truncated inexact Newton-like methods \cite{Pratt:1998:GNF,Metivier:2013:FWI}, and may also be used without Tikhonov regularization by progressively increasing the dimension $K$ of the spectral basis with frequency. In \cite{GN2019}, the AEI algorithm was extended to multi-parameter inverse medium
problems, including the well-known layered Marmousi subsurface model from geosciences. 
Recently, it was extended to electromagnetic inverse scattering problems at fixed frequency \cite{BD2017} and also to time-dependent
inverse scattering problems when the illuminating source is unknown~\cite{GGNA2019}.

So far, the remarkable efficiency and accuracy of the AS decomposition for the approximation of piecewise constant functions is only justified via numerical evidence. Although previous inversion algorithms based on the AS decomposition iteratively adapt the basis functions $\vphi_k$, they do not provide any criteria for adapting the dimension $K$ of the search space.
Here, we precisely address these two open questions. 
In Section \ref{sec:InverseProblem}, we present a strategy for adapting the dimension $K$ of the search space,
by solving a small quadratically constrained quadratic minimization problem to filter basis functions while preserving important features. 
The resulting new ASI (adaptive spectral inversion) algorithm is listed in Section 2.2.
In Section \ref{sec:AE_analysis}, we derive rigorous error estimates for $\vphi_0$ and for the first eigenvalues and eigenfunctions of the elliptic operator associated with a piecewise constant function. In particular, we prove that $\vphi_0$ and the first eigenfunctions $\vphi_k$ are ``almost'' piecewise constant in the sense that their gradients are small outside a neighborhood of internal discontinuities.
We also provide a numerical example which illustrates the theory.
In Section \ref{sec:NumRes}, we apply the ASI Algorithm to two inverse medium problems, 
the first where the medium is composed of five simple geometric inclusions, and the second, where 
the medium corresponds to a two-dimensional model of a salt dome from geophysics.
Finally, we conclude with some remarks in Section \ref{sec:conc}.

%
%
%
%
%
%
%
%
%
%
%
%
%
%
%
%
\section{Inverse scattering problem\label{sec:InverseProblem}}
First, we consider a time-harmonic inverse scattering problem and reformulate it as a PDE-constrained optimization problem. Then we present the adaptive spectral inversion (ASI)  method, list the full ASI Algorithm and discuss in further detail the individual steps in adapting both the basis and its dimension during the nonlinear
iteration. 

\subsection{Inverse scattering problem}
We consider a time-harmonic inverse scattering problem where the scattered wave field, $y(x)$, satisfies the Helmholtz equation:
\begin{subequations}
  \label{eq:HE}
  \begin{alignat}{4}
  -\dvg\rb{ \uTotal(x) \nabla y(x)} \ - \ \omega^2\ y(x) &\ = \ f(x)      ,&&\qquad  x\in\Omega,
  \label{eq:HE_PDE}
  \\
  \frac{\partial y}{\partial n}(x) \ -\  i \frac{\omega}{\sqrt{\uTotal(x)}}\ y(x) &\ = \ g(x) ,&&\qquad x\in\partial \Omega.
  \label{eq:HE_ABC}
   \end{alignat}
\end{subequations}
Here, $\Omega\subset \mathbb{R}^d$ is a bounded domain with Lipschitz boundary $\partial \Omega$
and outward unit normal $n$,  $f\in L^2(\Omega)$ and $g \in L^2(\partial \Omega)$ are known sources,
and $\omega=2\pi\nu$ is the angular frequency corresponding to the (regular) frequency $\nu>0$.
The squared wave speed, $\uTotal(x)$, of the (unknown) medium satisfies $\uTotal(x)\ge u_{\min}>0$
 throughout  $\Omega$
and is assumed known on the boundary $\partial\Om$.

Given observations $\yObs_\ell$ on a subset $\G$ of $\p\Om$ of the scattered fields, $\y_\ell=\y_\ell[u]$,  due
to known sources $f=f_\ell$ and $g=g_\ell$, $\ell=1,\ldots,N_s$,
we wish to recover the medium $\uTotal(x)$ inside $\Omega$ by minimizing the cost functional
$\cJ:L^\infty(\Omega)\rightarrow \mathbb{R}$,
\begin{equation}
	\label{eq:cost_functional}
	\cJ[w] = \frac{1}{2}\sum_{\ell=1}^{N_S}\|\y_\ell[w]-\yObs_\ell \|_{L^2(\Gamma)}^2.
\end{equation}
Hence, we consider the PDE-constrained optimization problem,
\begin{equation}
	\label{eq:optimization-problem}
	\uTotalMin = \argmin_{w\in W} \ \cJ[w] ,
\end{equation}
where $W$ is a space of candidate media 
and  $y_\ell$ solves \deq{eq:HE} with $f=f_\ell$ and $g=g_\ell$.
We regard $W$ as a subspace of $L^2(\Om)$ with its standard inner product and norm, denoted by $\angb{\cdot,\cdot}$ and $\|\cdot\|$, respectively.

\subsection{Adaptive spectral inversion \label{sec:AS_algorithm}}


The adaptive spectral inversion (ASI) algorithm minimizes $\cJ$ in (\ref{eq:cost_functional}) 
over a finite-dimensional subspace of $W$ by building and repeatedly adapting a finite basis as follows.
Suppose we have, at iteration $m$, a function $\vphi^{(m)}_0\in W$, which coincides with the true medium $u$ on the boundary, and a subspace $\Psi_{}^{(m)} \subset W$ spanned by an orthonormal set
\begin{equation}
    \Psi_{}^{(m)} = \Span\{\psi_1^{(m)},\ldots,\psi_{\PsiDim_{m}}^{(m)}\} 
\end{equation}
of functions that vanish on $\partial \Omega$.
Then, we determine a (local) minimizer $u^{(m)}$ of $\cJ$ in the $\PsiDim_{m}$-dimensional affine space $\vphi_0^{(m)}+\Psi_{}^{(m)}$, i.e.,
\begin{equation}
	\label{eq:optimization-problem_m}
	u^{(m)} = \argmin_{v \in \vphi_0^{(m)}+\Psi_{}^{(m)}} \cJ [v] ,
\end{equation}
using any standard Newton or quasi-Newton method; the smaller  $\PsiDim_{m}$, the cheaper the
numerical solution of the nonlinear optimization problem \eqref{eq:optimization-problem_m}.

%
%

Once we have determined $u^{(m)}$, we must update the search space needed at the next iteration, that is, determine $\vphi_{0}^{(m+1)}\in W$ and $\Psi^{(m+1)}\sbt W$.
First, we compute $\vphi_{0}^{(m+1)}$, by solving
\begin{equation}\label{eq:phi0BVP_alg_m}
	L_\veps\big[u^{(m)}\big]\vphi_0^{(m+1)} =0
	\quad \text{in $\Omega$,}
	\qquad
        \vphi_0^{(m+1)} = u
	\quad \text{on $\partial\Omega$,}
\end{equation}
where $L_\veps[w]$ is a $w$-dependent, symmetric and elliptic operator; the particular form used here is specified below, but other choices are possible -- see Section \ref{sec:AE_analysis} and \cite{GN2019}.
Recall that $u$ is assumed known on the boundary $\partial\Omega$. 
Then, we compute the first $K_{m+1}$ (orthonormal) eigenfunctions $\vphi_{1}^{(m+1)},\ldots,\vphi_{K_{m+1}}^{(m+1)}$ of $L_\veps[ u^{(m)}-\vphi_0^{(m+1)} ]$ (or possibly $L_\veps[ u^{(m)}]$, cf. \cite{GKN2017,GN2019}) by solving the eigenvalue problem
\begin{equation}\label{eq:eigenValProb_alg_m}
	L_\veps\big[u^{(m)}-\vphi_0^{(m+1)} \big]\vphi_k^{(m+1)} =\lam_k \vphi_k^{(m+1)}
	\quad \text{in $\Omega$,}
	\qquad
	\vphi_k^{(m+1)} =0
        \quad \text{on $\partial\Omega$,}
\end{equation}
 for the smallest, not necessarily distinct eigenvalues $0<\lam_1\le \cdots\le\lam_{K_{m+1}}$.
The eigenfunctions $\vphi_1^{(m+1)},\ldots,\vphi_{K_{m+1}}^{(m+1)}$ should enable
an efficient representation of the current iterate $u^{(m)}$ but also enrich the search space by introducing new potential candidate media for further minimizing  $\cJ$ at the next iteration.

Next, we merge the current search space $\Psi^{(m)}$ with the space spanned by the new eigenfunctions as
\begin{equation}
    \label{eq:merged-space}
	\wt{\Psi}^{(m+1)} = \Psi^{(m)} +
		\Span\big\{\vphi_0^{(m+1)}-\vphi_0^{(m)},\vphi_1^{(m+1)},\ldots,\vphi_{K_{m+1}}^{(m+1)} \big\} 
\end{equation}
to ensure that the current solution $u^{(m)} \in \vphi_0^{(m)}+\Psi_{}^{(m)}$ also belongs to 
$\vphi_0^{(m+1)} + \wt{\Psi}^{(m+1)}$. By a standard Gram-Schmidt procedure, we 
compute an $L^2$-orthonormal basis
for the merged space,
\[
	\wt{\Psi}^{(m+1)} =\Span\big\{ \psi_1^{(m+1)},\ldots,\psi_{\wt{\PsiDim}_{m+1}}^{(m+1)} \big\}.
\]

In the process of adapting the search space, it is crucial to also adapt its new dimension, 
which not only directly impacts the computational cost but also acts as regularization.
Hence, we shall reduce $\wt{\Psi}^{(m+1)}$ and thus
obtain the new search space $\Psi^{(m+1)} \subset \wt{\Psi}^{(m+1)}$
of dimension $\PsiDim_{m+1}$ while imposing some regularity on the solution.
To decide which basis functions of  $\wt{\Psi}^{(m+1)}$ to keep and which to discard, we compute an indicator $\widetilde{u}^{(m+1)}\in \wt{\Psi}^{(m+1)}$ as a filtered approximation of $u^{(m)}-\vphi_0^{(m+1)}$ in $\wt{\Psi}^{(m+1)}$. Then, we reorder the basis functions $\psi^{(m+1)}_k$ in decreasing order of its Fourier coefficients
\begin{equation}\label{eq:FourierCoeff}
	\g_k = \angb{\widetilde{u}^{(m+1)}, \psi^{(m+1)}_k} ,
\end{equation}
and discard any $\psi^{(m+1)}_k$ associated with small $\g_k$. This
yields the new search space $\Psi^{(m+1)}$ 
spanned by the remaining $\PsiDim_{m+1}$ basis functions,
\begin{equation}
	\label{eq:reduced-and-trimmed-space}
	\Psi^{(m+1)} = \Span\big\{ \psi_1^{(m+1)},\ldots,\psi_{\PsiDim_{m+1}}^{(m+1)} \big\} .
\end{equation}
Below we summarize the entire adaptive spectral inversion (ASI) Algorithm:
\begin{algorithm}\caption{ASI Algorithm}\label{algo:ASI}
\DontPrintSemicolon
\KwIn{initial guess $u^{(0)}$ of $u$, affine space $\vphi_0^{(1)} + \Psi^{(1)}$}
\KwOut{reconstruction $u^{(m)}$ of $u$}
\For{$m=1,2,\ldots$}{
	Determine (local) minimizer $\uTotal^{(m)}$ of $\cJ[v]$ in the $\PsiDim_{m}$-dimensional affine space $\vphi_0^{(m)} + \Psi^{(m)}$ by solving \eqref{eq:optimization-problem_m}.
	\;\BlankLine
	\label{algo:STEP-1}

	\uIf{$\|u^{(m)}-u^{(m-1)}\|<\varepsilon_{tol}$}
	{
		{\bf  return} $u^{(m)}$
	}
	\BlankLine

	Compute $\vphi_0^{(m+1)}$ and the first eigenfunctions $\vphi_1^{(m+1)},\ldots,\vphi_{K_{m+1}}^{(m+1)}$ as in \eqref{eq:phi0BVP_alg_m}, \eqref{eq:eigenValProb_alg_m} for the current medium $u^{(m)}$.
	\;\BlankLine
	\label{algo:STEP-2}

	\textit{Merge:}\ \ Compute an orthonormal basis $\{\psi_k^{(m+1)}\}_{k=1}^{\wt{\PsiDim}_{m+1}}$ for the merged space $\wt{\Psi}^{(m+1)}$ given by \deq{eq:merged-space}. 
	\;\BlankLine
	\label{algo:STEP-3}

	\textit{Filter:}\ \ Compute the indicator $\widetilde{u}^{(m+1)}$ as a filtered (regularized) approximation of $u^{(m)}-\vphi_0^{(m+1)}$ in $\wt{\Psi}^{(m+1)}$. 
	\;\BlankLine
	\label{algo:STEP-4}

	\textit{Truncate:}\ \ Reduce the dimension of $\wt{\Psi}^{(m+1)}$ by discarding those $\psi_i^{(m+1)}$ with smallest Fourier coefficients \deq{eq:FourierCoeff}. 
	This yields $\Psi^{(m+1)}$ of dimension $\PsiDim_{m+1}$.
	\;
	\label{algo:STEP-5}

}
\end{algorithm}
In the above ASI Algorithm, most of the computational work occurs in Step \ref{algo:STEP-1}, which involves multiple numerical solutions of the (forward) problem (\ref{eq:HE}). 
For \eqref{eq:phi0BVP_alg_m} and \eqref{eq:eigenValProb_alg_m}
in Step \ref{algo:STEP-2}, we always use $L_\veps[w]$ of the form
\begin{equation}\label{eq:linear_op}
	L_\veps[w]v=-\dvg\rb{\mu_{\veps}[w]\nabla v} ,
\end{equation}
where the $w$-dependent weight function $\mu_\veps[w]$ is given by
\begin{equation}\label{eq:mutwo}
	\mu_{\veps}[w]=\frac{1}{\sqrt{|\nabla w|^2+\veps^2}} \, ,
\end{equation}
with $\veps>0$ a small parameter; in our computations, we always set  $\veps = 10^{-8}$.
Note that the theoretical properties of the adaptive spectral decomposition proved in Section \ref{sec:AE_analysis} also apply to more general $\mu_\veps[w]$. We now describe in more detail the individual steps of the above ASI Algorithm. 

To initialize the ASI Algorithm, we require an initial guess $u^{(0)}$,
an approximation of the background $\vphi_0^{(1)}$ that coincides with the medium on the boundary,
and a finite dimensional space $\Psi_{}^{(1)}$ given as the span of an orthonormal basis
\[
	\Psi_{}^{(1)} = \Span\big\{\psi_1^{(1)},\ldots,\psi_{\PsiDim_1}^{(1)} \big\} .
\]
The initial background $\vphi_{0}^{(1)}$ and orthonormal basis of $\Psi_{}^{(1)}$ can
be determined by solving (\ref{eq:phi0BVP_alg_m}) and (\ref{eq:eigenValProb_alg_m}) with $m=0$ and setting $\psi^{(1)}_k=\vphi_k^{(1)}$. For instance, if $u^{(0)}$ is constant throughout $\Omega$ and equal to $u$ on the boundary, then $\vphi_{0}^{(1)} = u^{(0)}$ whereas $\vphi_1^{(1)}, \vphi_2^{(1)}, \dots$  correspond to the first eigenfunctions of the Laplacian.
Alternatively, one may choose $\vphi_{0}^{(1)}$ or the basis of $\Psi_{}^{(1)}$ a priori,
independently of $u^{(0)}$.

For the minimization of (\ref{eq:optimization-problem_m}) in Step \ref{algo:STEP-1}, we opt for a quasi-Newton method with rank-two updates (BFGS, \cite{NW2006}), but other choices are possible. As initial guess
for $m>1$, one can use the $L^2$-projection onto $\vphi_0^{(m)}+\Psi^{(m)}$ either of $u^{(m-1)}$ or its filtered approximation $\vphi_0^{(m)}+\widetilde{u}^{(m)}$.

In Step \ref{algo:STEP-2}, the number of new eigenfunctions $K_{m+1}$ is somewhat arbitrary,
since the dimension of the new basis is anyway adapted subsequently; 
here, we typically take $K_{m+1}= \PsiDim_{m}$ with $J_1=100$.
To obtain the $L^2$-orthonormal basis $\{\psi_i^{(m+1)}\}_{i=1}^{\wt{\PsiDim}_{m+1}}$ for $\wt{\Psi}^{(m+1)}$ in Step \ref{algo:STEP-3}, we apply the standard modified Gram-Schmidt algorithm to the ordered set
\[
	\Big(\vphi_1^{(m+1)},\ldots,\vphi_{K_{m+1}}^{(m+1)},\vphi_0^{(m+1)}-\vphi_0^{(m)},
		\psi_1^{(m)},\ldots,\psi_{\PsiDim_{m}}^{(m)} \Big)
\]
spanning $\wt{\Psi}^{(m+1)}$.

In Step \ref{algo:STEP-4}, we compute the indicator $\widetilde{u}^{(m+1)}$ used subsequently for truncating the space $\wt{\Psi}^{(m+1)}$ as follows.
The basis of the truncated new search space $\Psi^{(m+1)}$ ought to preserve edges in the medium but also suppress noise. Hence, we seek an indicator which is close to $u^{(m)}-\vphi_0^{(m+1)}$, yet with minimal TV,
by considering the  constrained minimization problem
\begin{equation}
\begin{aligned}
	\label{eq:ConsMinTV}
	\min_{v\in \wt{\Psi}^{(m+1)}}\ & \int_{\Om} |{\nabla v}|,
			\quad \text{subject to} \\
		& \|v-(u^{(m)}-\vphi_0^{(m+1)})\|^2 \le \veps_\Psi^2 \|u^{(m)}-\vphi_0^{(m+1)}\|^2 ,
\end{aligned}
\end{equation}
for a prescribed tolerance $\varepsilon_{\Psi}>0$.
Since the solution of \eqref{eq:ConsMinTV} is expensive, we replace 
the TV functional
\[
	\int_{\Om} |{\nabla v}| ,
\]
by the approximation \cite{GKN2017,GN2019}
\begin{equation}\label{eq:AppReg_quadForm}
	\int_{\Om}\mu_\veps[u^{(m)}-\vphi_0^{(m+1)}] |{\nabla v}|^2.
\end{equation}
Hence, we compute the indicator $\widetilde{u}^{(m+1)}$ by solving the quadratically constrained quadratic minimization problem for fixed $w =u^{(m)}-\vphi_0^{(m+1)}$,
\begin{equation}
\begin{aligned}\label{eq:Filtered_u}
	\widetilde{u}^{(m+1)} =\argmin_{v\in \wt{\Psi}^{(m+1)}}\
		& \int_{\Om}\mu_\veps[w] |{\nabla v}|^2 ,
			\quad \text{subject to} \\
		& \|v-w\|^2 \le \veps_\Psi^2 \|w\|^2 .
\end{aligned}
\end{equation}
which is cheap.

Once $\widetilde{u}^{(m+1)}$ has been computed,
we truncate  $\wt{\Psi}^{(m+1)}$ in Step \ref{algo:STEP-5} as follows.
Given the Fourier coefficients $\gamma_\ell$ of $\widetilde{u}^{(m+1)}$ arranged in decreasing order, with $\gamma_\ell$ as in \deq{eq:FourierCoeff}, we remove the maximal number of 
basis functions with corresponding smallest Fourier coefficients, setting
\begin{equation}
	\label{eq:truncation-criterion}
	\PsiDim_{m+1}^0 = \min
  		\bigg\{ J\in\{1,\ldots,\wt{\PsiDim}_{m+1}\}\ :\
			\sum_{\ell=\PsiDim+1}^{\wt{\PsiDim}_{m+1}}{\gamma_\ell^2}
				\leq \varepsilon_{\Psi}^2 |\gamma|^2
		\bigg\}, \qquad \rho = \frac{\PsiDim_{m+1}^0}{\PsiDim_{m}}
\end{equation}
for a given tolerance  $\varepsilon_{\Psi}$.
To avoid abrupt changes in the number of basis functions from one iteration to the next, 
we accept the value $\PsiDim_{m+1}=\PsiDim_{m+1}^0$ without further change  if 
\begin{equation}
	\label{eq:truncation-criterion-lowerbound}
	\rho \in [\rho_0, \rho_1] ,
\end{equation}
for some prescribed $0\le \rho_0 \le 1 \le \rho_1$. Otherwise if
$\rho > \rho_1$, we still set $\PsiDim_{m+1}=\PsiDim_{m+1}^0$ but also increase $\varepsilon_\Psi$, e.g., by doubling it, to avoid such an overly large increase in dimension at the next iteration. 
On the other hand if $\rho< \rho_0$, we set $\PsiDim_{m+1} = \left \lceil{\rho_0 \PsiDim_{m}}\right \rceil$ 
to avoid an overly rapid drop in dimension and thus lose important information; 
moreover, we decrease $\varepsilon_\Psi$ in (\ref{eq:truncation-criterion}), e.g., by halving it.
Typically, we choose $\rho_0=0.9$, $\rho_1=1.1$, and $\varepsilon_\Psi=10^{-3}$.

%

Finally, we embed the ASI algorithm in a standard frequency stepping approach, where the inverse problem is solved at increasingly higher frequencies $\nu$.
Whenever two consecutive iterates satisfy
\begin{equation}
    \label{eq:frequency-stepping-criterion}
    \|u^{(m)}-u^{(m-1)}\|<\varepsilon_{\nu} ,
\end{equation}
for a given tolerance $\varepsilon_{\nu}$, we proceed to the observations obtained at a higher frequency.

%
%
%
%
%
\section{Analysis of adaptive spectral decompositions}
\label{sec:AE_analysis}

In this section we derive estimates for $\vphi_0$ and the eigenvalues $(\lam_k)_k$ and eigenfunctions $(\vphi_k)_k$ of the linear operator $L_\veps[u_\del]$ defined in \deq{eq:linear_op} with $u_\del$ an $H^1$-regular approximation of a given piecewise constant function $u$.
Typically, the medium-dependent weight function $\mu_\veps[w]$ which characterizes the AS decomposition has either the form
\begin{equation}\label{eq:muq}
	\mu_\veps[w](x) =\ \frac{1}{(|\nabla w(x)|^q+\veps^q)^{1/q}} \, ,
\end{equation}
for some $q\in[1,\infty)$, or
\begin{equation}\label{eq:muinf}
	\mu_\veps[w](x)  =\ \frac{1}{\max(|\nabla w(x)|,\, \veps )} \, .
\end{equation}
For the analysis below, we allow $\mu_\veps[w]$ to have a more general form.

\subsection{Assumptions and definitions}
\subsubsection{Medium dependent weight function}

Let $\Om\sbt \R^d$ be a bounded and connected Lipschitz domain.
We assume that for $w:\Om\to\R$, with $\nabla w\in L^{\infty}(\Om)$, $\mu_\veps[w]$ is given by
\begin{equation}\label{eq:mu_hatmu}
	\mu_\veps[w](x)=\hat\mu_\veps(|\nabla w(x)|) ,
	\qquad
	x\in\Om ,
\end{equation}
where $\hat\mu_\veps:[0,\infty)\to\R$ is a non-increasing function that satisfies
\begin{equation}\label{eq:mu_cond}
	\hat\mu_\veps(0)=\veps^{-1} ,
	\qquad
	0<\hat\mu_\veps(t) ,
	\quad
	t \hat\mu_\veps(t)\le 1 ,
	\quad
	t\ge 0 .
\end{equation}
In particular, these assumptions yield
\begin{equation}\label{eq:REstMu}
	\mu_\veps[w](x)|\nabla w(x)| \le 1 ,
	\qquad
	\text{a.e. $x\in\Om$,}
\end{equation}
and
\begin{equation}\label{eq:LEstMu}
	0 < \hat\mu_\veps(\|\nabla w\|_{L^\infty(\Om)}) \le \mu_\veps[w](x)
	\qquad
	\text{a.e. $x\in\Om$.}
\end{equation}
Although this framework encompasses weight functions as in \deq{eq:muq} or \deq{eq:muinf}, it does not include weight functions, associated with the Lorentzian or Gaussian penalty terms \cite{GN2019}, for instance, which do not satisfy the last inequality of \deq{eq:mu_cond}.
However, the above framework easily extends to more general weight functions that satisfy an estimate
\begin{equation}\label{eq:mu_cond_r}
	t^r \hat\mu_\veps(t)\le 1 ,
	\qquad
	t\ge 0 ,
\end{equation}
for some $r\in[1,2]$.
This extension, indeed, addresses the Lorentzian and Gaussian weight functions which satisfy \deq{eq:mu_cond_r} with $r=2$.
For the analysis, we assume \deq{eq:mu_cond}, i.e., \deq{eq:mu_cond_r} with $r=1$, and comment on the extension to the more general case $r\in[1,2]$ in Remark \ref{rem:generalization} below.

\subsubsection{Elliptic boundary value problems}

To include FE approximations in the analysis, we formulate the boundary value problems \deq{eq:phi0BVP} and \deq{eq:eigenValProb} in closed subspaces $\cV^\del\sbt H^1(\Om)$ and $\cV^\del_0=\cV^\del \cap H^1_0(\Om)$, respectively.
We say that $\vphi_0\in\cV^\del$ satisfies
\begin{equation}\label{eq:phi0BVP_analysis}
	L_\veps[u_\del]\vphi_0 =0
	\quad \text{in $\Omega$,}
	\qquad
        \vphi_0 = u_\del
	\quad \text{on $\partial\Omega$}
\end{equation}
in $\cV^\del_0$, if
\bseql{eq:phi0BVP_weakForm}
\begin{align}
	B[\vphi_0,v] &\ = \ 0 , \qquad \forall \, v\in\cV^\del_0 \\
	\vphi_0      &\ = \ u_\del , \qquad \text{on $\partial\Omega$,}
\end{align}
\eseq
where $B[\cdot,\cdot]$ is the bilinear form given by
\begin{equation}\label{eq:eigenValProb_bilinearform}
	B[v,w]
	\ = \ \angb{\mu_{\veps}[u_\del] \nabla v, \nabla w},
\end{equation}
with $\angb{\cdot,\cdot}$ denoting the standard $L^2(\Om)$ inner product.
We say that $\lam\in\R$ is an eigenvalue of $L_\veps[u_\del]$ in $\cV^\del_0$, if there exists $0\ne\vphi\in\cV^\del_0$ such that
\begin{equation}\label{eq:eigenValProb_analysis}
	L_\veps[u_\del]\vphi =\lam \vphi
	\quad \text{in $\Omega$,}
	\qquad
	\vphi =0
        \quad \text{on $\partial\Omega$,}
\end{equation}
in $\cV^\del_0$, that is, if 
\begin{equation}\label{eq:eigenValProb_weakForm}
	B[\vphi,v]
	\ = \ \lam\angb{\vphi,v},
	\qquad
	\forall \, v\in \cV^\del_0 .
\end{equation}

As shown below, the operator $L_\veps[u_\del]$ is symmetric and uniformly elliptic.
Thus, we let $(\lam_k)$ be the nondecreasing (perhaps finite) sequence of (real positive) eigenvalues of $L_\veps[u_\del]$ in $\cV^\del_0$, with each eigenvalue repeated according to its multiplicity, and $(\vphi_k)$ be a basis of $\cV^\del_0$ of corresponding eigenfunctions, orthonormal with respect to the $L^2$ inner product.
Finally, we denote by $\|\cdot\|$ the norm of $L^2(\Om)$, and by $|\cdot|$ the $\ell^2$-norm.

\subsubsection{Piecewise constant medium}

\begin{figure}
\centering
	\includegraphics[width=11cm]{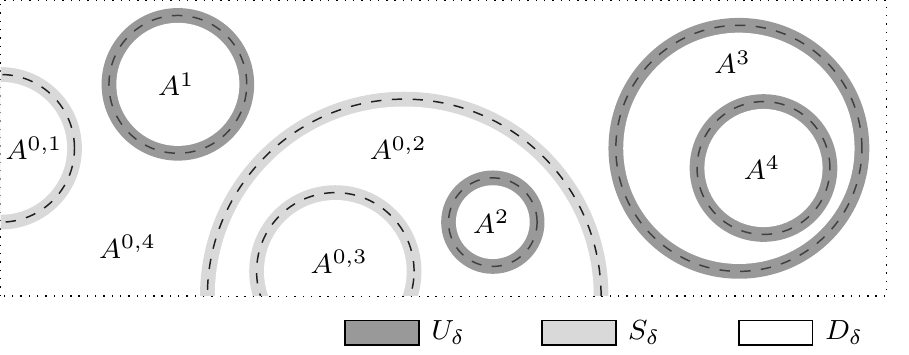}
	\caption{Typical arrangement in two dimensions of the sets $A^k$, $A^{0,m}$ composing the piecewise constant medium $u$, and $U_\del$, $S_\del$, and $D_\del$ given by \deq{eq:setU}, \deq{eq:setS_del}, and \deq{eq:setD}, respectively.}
    \label{fig:1D-Illustration}
\end{figure}

Let the medium $u:\Om\to\R$ be piecewise constant admitting a decomposition
\begin{equation}\label{eq:u_decomp}
	u(x)
	\ = \ u^0(x)+\wt{u}(x) ,
	\qquad
	x\in\ol{\Om}
\end{equation}
into a piecewise constant background $u^0$, composed of features reaching the boundary $\p\Om$, and a piecewise constant perturbation $\wt u$, composed of a finite number of inclusions separated from the boundary.
More precisely, we assume
\begin{equation}
	u^0
	\ = \ \sum_{m=1}^{M} \hat u^{0,m}\chi^{0,m}
\end{equation}
where $\hat u^{0,m}\in\R$, and $\chi^{0,1},\ldots,\chi^{0,M}$ are the characteristic functions of mutually disjoint connected open sets $A^{0,1},\ldots,A^{0,M}$ such that
\begin{equation}
	\ol{\Om} \sbt \ol{\bigcup_{m=1}^M A^{0,m} } ,
	\qquad
	\cH^{d-1}\brb{\p A^{0,m}\cap \p\Om}>0 ,
\end{equation}
where $\cH^{d-1}$ denotes the $(d-1)$-dimensional Hausdorff measure.
Although the coefficients $\hat u^{0,k}$ must satisfy $\hat u^{0,k}>0$ for the Helmholtz equation, they may take any real value in the present analysis.
We further assume that
\begin{equation}\label{eq:wtu}
	\wt{u}
	\ = \ \sum_{k=1}^{K} \hat u^{k}\chi^{k} ,
	\qquad
	\hat u^{k}\ne0 ,
\end{equation}
where for each $k=1,\ldots,K$, $\chi^k$ is the characteristic function of a connected open set $A^k\sbt\sbt\Om\setminus S$, where
\begin{equation}\label{eq:setS}
	S
	\ = \ \rb{\bigcup_{m=1}^{M} \p A^{0,m}}\setminus \p\Om
\end{equation}
is the set of jump discontinuities, or interfaces, in the background $u^0$.
Moreover, we assume that the sets $A^1,\ldots,A^K$ have mutually disjoint boundaries; the full inclusion of one set inside another, however, is allowed.

For $\del>0$, let $S_\del$ denote the $\del$-neighborhood of $S$,
\begin{equation}\label{eq:setS_del}
	S_\del
	\ = \ \curb{x\in\Om\, |\ \dist(x,S)<\del} ,
\end{equation}
and, similarly, let $U^k_\del$ denote the $\del$-neighborhood of $\p A^k$,
\begin{equation}\label{eq:setU}
	U^k_\del
	\ = \ \curb{x\in\Om\, |\ \dist(x,\p A^k)<\del} ,
	\quad\text{and}\quad
	U_\del
	\ = \ \bigcup_{k=1}^M U^k_\del .
\end{equation}
Then we define the open complement $D_\del$ as
\begin{equation}\label{eq:setD}
	D_\del
	\ = \ \Om\setminus \rb{\ol{U_\del\cup S_\del}} ,
\end{equation}
the $\del$-interior $A^k_\del$ of $A^k$ as
\begin{equation}\label{eq:setA}
	A^k_\del
	\ = \ \rb{\Om\setminus \ol{U^k_\del}}\cap A^k ,
	\quad\text{and}\quad
	A_\del
	\ = \ \bigcup_{k=1}^K A^k_\del .
\end{equation}
Figure \ref{fig:1D-Illustration} shows a typical arrangement in two dimensions.

Let $u_\del$ be an approximation of $u$ obtained by a linear method.
For example, $u_\del$ may be the interpolant of $u$ in an $H^1$-conforming FE space;
then, the parameter $\del$ corresponds to the mesh size.
We assume that for each $m=1,\ldots,M$, the approximations $\chi^{0,m}_\del\ \in \cV^\del$ of $\chi^{0,m}$ satisfy $\lim_{\del\to0}\chi_\del^{0,m}=\chi^{0,m}$ in $L^2(\Om)$, 
and similarly, for each $k=1,\ldots,K$, that the approximations $\chi^{k}_\del\ \in \cV^\del_0$ of $\chi^{k}$ satisfy $\lim_{\del\to0}\chi_\del^{k}=\chi^{k}$ in $L^2(\Om)$.
For each $\del>0$, the $H^1$-regular approximation $u_\del$ is thus given by
\begin{equation}\label{eq:u_del_decomp}
	u_\del \ = \ u^0_\del+\wt{u}_\del ,
\end{equation}
where
\begin{equation}
	u^0_\del
	\ = \ \sum_{m=1}^{M} \hat u^{0,m}\chi^{0,m}_\del \in\cV^\del ,
	\qquad
	\wt{u}_\del
	\ = \ \sum_{k=1}^{K} \hat u^{k}\chi^{k}_\del \in\cV^\del_0 .
\end{equation}

\subsection{Statement of main results and discussion}

To simplify the presentation, we include in this section only short proofs and proofs of the main results; the remaining proofs are provided in Section \ref{sec:MainProofs}.
The main result, given by Theorem \ref{thm:main}, provides estimates for the approximation $\vphi_0$ of the background $u^0$, and for the first $K$ eigenvalues $\lam_k$ and eigenfunctions $\vphi_k$ of $L_\veps[u_\del]$ in \deq{eq:linear_op}.
From Theorem \ref{thm:main}, we deduce Corollary \ref{cor:main} which provides similar estimates for finite element formulations or for $u_\del$ obtained from the convolution of $u$ with a mollifier.

Theorem \ref{thm:main} relies on Lemmas \ref{lem:background_estimation} and \ref{lem:eigVal_est} which require the approximation $\{\chi_\del\}_\del$ of each characteristic function $\chi=\chi^k$, and $\chi=\chi^{0,m}$, to satisfy
\begin{equation}\label{eq:TV_Bounded}
	\|\nabla \chi_\del\|_{L^1(\Om)}\le C ,
\end{equation}
for every $\del>0$ sufficiently small.
Note that since $\chi_\del$ converges to a function $\chi$ with jump discontinuities, the gradients of $\chi_\del$ need not be bounded uniformly with respect to $\del$, for $\del$ close to zero.
Whether \deq{eq:TV_Bounded} is satisfied depends on geometric properties of the method by which $\chi_\del$ is obtained, as well as on properties of the set $A=\chi^{-1}(\{1\})$.
The following lemma provides sufficient conditions for \deq{eq:TV_Bounded} to hold.

\begin{lemma}\label{lem:SuffCondTVEst}
Let $A\sbt\R^d$ be a bounded Lipschitz domain,
\[
	U_\del
	\ = \ \curb{x\in\R^d\, |\ \dist(x,\p A)<\del}
\]
with $\del\in(0,\eta]$, for $\eta>0$, and  $\cL$ the Lebesgue measure.
Then, the following assertions hold:
\begin{enumerate}
\item\label{lem:SuffCondTVEst_L}
There exists a constant $C>0$ such that $\del^{-1}\cL(U_\del)<C$, for every $\del\in(0,\eta]$.

\item\label{lem:SuffCondTVEst_g}
If $\{g_\del\}_{\del\in(0,\eta]}\sbt L^p(\R^d)$ such that $\supp{g_\del} \sbt \ol{U_\del}$ for all $\del\in(0,\eta]$, and if
\begin{equation}\label{eq:CondLpEst}
	\del^{1-1/p}\|g_\del\|_{L^p(\R^d)}\le C_1 ,
	\qquad
	\forall \, \del\in(0,\eta]
\end{equation}
for some $p\in[1,\infty]$ (with the usual convention $1/\infty:=0$), then there exists a constant $C>0$, such that for every $\del\in(0,\eta]$,
\begin{equation}
	\|g_\del\|_{L^1(\R^d)}\le C .
\end{equation}
\end{enumerate}
\end{lemma}

In particular, for $p=\infty$, \deq{eq:CondLpEst} reduces to $\del |g_\del|\le C_1$ a.e.\ in $\R^d$, for all $\del\in(0,\eta]$.
Also note that for $p=1$, the conclusion of the lemma is trivial.

In the following, we assume $\eta>0$ sufficiently small such that
\begin{equation}\label{eq:USetsSep}
	A^k_\eta \ne \emptyset
		\quad
		\forall\, k ,
	\qquad
	\ol{S_\eta}\cap \ol{U_\eta} =\emptyset ,
	\qquad
	\ol{U^k_\eta}\cap \ol{U^j_\eta} = \emptyset
		\quad
		\forall\, k\ne j ,
\end{equation}
and for each connected component $E_\eta$ of $D_\eta\setminus A_\eta$, there holds $\cH^{d-1}(\p E_\eta \cap \p\Om)>0$.
In other words, we assume for $\del\in(0,\eta]$, that all $\del$-interiors $A^k_\del$ of $A^k$ are non-empty, that all $\del$-neighborhoods of the interfaces of the medium, $S_\del$ and $U^k_\del$, do not intersect, and that the only parts of the open complement, $D_\del$, isolated from the boundary $\p\Om$ are the $\del$-interiors $A^k_\del$ of the inclusions $A^k$.
Since the boundaries of $A^1,\ldots,A^K$ are mutually disjoint, and $A^k\sbt\sbt\Om\setminus S$, $k=1,\ldots,K$, such an $\eta>0$ exists.
We further suppose that for each $m=1,\ldots,M$,
\begin{subequations}\label{eq:approx_reg0_reg}
\begin{equation}\label{eq:approx_reg0}
	\nabla\chi^{0,m}_\del\in L^\infty(\Om),
	\quad
	\supp{\nabla\chi^{0,m}_\del} \sbt \ol{S_\del},
	\qquad
	\forall\, \del\in(0,\eta],
\end{equation}
and that for each $k=1,\ldots,K$,
\begin{equation}\label{eq:approx_reg}
	\nabla\chi^k_\del\in L^\infty(\Om),
	\quad
	\supp{\nabla\chi^k_\del} \sbt \ol{U^k_\del},
	\qquad
	\forall\, \del\in(0,\eta].
\end{equation}
\end{subequations}
These assumptions imply that
\begin{equation}\label{eq:grad_u0_tildeu_supp}
	\supp{\nabla u^{0}_\del} \sbt \ol{S_\del},
	\qquad
	\supp{\nabla \wt{u}_\del} \sbt \ol{U_\del},
\end{equation}
and 
\begin{equation}\label{eq:mu_eps_D_del}
	\mu_\veps[u_\del]
	\ = \ \veps^{-1}
	\qquad\text{a.e.\ in $D_\del$.}
\end{equation}
By \deq{eq:approx_reg0_reg}, for each $\del\in(0,\eta]$, $\nabla u_\del\in L^\infty(\Om)$, and thus, due to \deq{eq:LEstMu} with $w=u_\del$, the operator $L_\veps[u_\del]$ is uniformly elliptic in $\Om$, for $\veps>0$.
A simple but useful conclusion we can draw from \deq{eq:mu_eps_D_del} is
\begin{equation}\label{eq:grad_D_del_B_est}
	\|\nabla \vphi\|_{L^2(D_\del)}^2 \le \veps B[\vphi,\vphi]
	\qquad
	\forall\, \vphi\in H^1(\Om).
\end{equation}

For the approximation $\vphi_0$ of the background $u^0$, we have the following estimate.

\begin{lemma}\label{lem:background_estimation}
For every $\veps>0$ and $\del\in(0,\eta]$, there holds
\begin{equation}
    \|\nabla \vphi_0\|_{L^2(D_{\delta})}^2
    \ \leq\   \varepsilon\, \|\nabla u^0_\del\|_{L^1(\Om)} .
\end{equation}
\end{lemma}

The estimates for the eigenfunctions $\vphi_j$ are based on the following simple result.

\begin{proposition}\label{prop:eigFunc_est0}
For every $\veps>0$ and $\del\in(0,\eta]$ there holds
\begin{equation}
	\|\nabla \vphi_j\|_{L^2(D_\del)}^2
	\ \leq\  \lam_j(\veps,\del)\, \veps
	\qquad
	\forall\, j\ge 1 .
\end{equation}
\end{proposition}

\begin{proof}
Since for each $j$, $\|\vphi_j\|=1$, the proposition follows from \deq{eq:grad_D_del_B_est} and \deq{eq:eigenValProb_weakForm}.
\end{proof}

Thus, to show that an eigenfunction $\vphi_k$ is ``almost'' piecewise constant, we need to estimate the corresponding eigenvalue $\lam_k$ of $L_\veps[u_\del]$ in $\cV^\del_0$, 
for which we rely on the following lemma.

\begin{lemma}\label{lem:eigVal_est}
There exists a constant $C$, independent of $\hat u^1,\ldots,\hat u^K$ and $\hat u^{0,1},\ldots,\hat u^{0,M}$ such that for every $\veps>0$, $\del\in(0,\eta]$, and $k=1,\ldots,K$, there holds
\begin{equation}\label{eq:ThmEst}
	\lam_k\ \le\  \frac{C\, |\tau|}{\min_{j} |\hat u^j|} ,
	\quad
	\tau=(\tau_k)\in\R^K ,
	\quad
	\tau_k \ =\ \tau_k(\del) \ =\  \|\nabla \chi^k_\del\|_{L^1(\Om)} .
\end{equation}
%
%
\end{lemma}

Together, the results above yield the following theorem.

\begin{theorem}\label{thm:main}
Let $u$ be given by \deq{eq:u_decomp}, the approximation, $u_\del$, of $u$ be given by \deq{eq:u_del_decomp}, $\vphi_0$ be given by \deq{eq:phi0BVP_weakForm}, and $(\lam_k,\vphi_k)$ with $k\ge 1$ satisfy \deq{eq:eigenValProb_weakForm}, where $(\lam_k)_k$ is non-decreasing and $(\vphi_k)_k$ orthonormal in $L^2(\Om)$.
Suppose $A^1,\ldots,A^K\sbt\Om$ and $A^{0,1},\ldots,A^{0,M}\sbt\Om$ have Lipschitz boundaries, and $\eta>0$ such that \deq{eq:USetsSep} is satisfied.
If \deq{eq:approx_reg0_reg} hold true, and there exists a constant $C$ such that for every $\del\in(0,\eta]$, each of the functions $\chi_\del=\chi^{0,m}_\del$, $m=1,\ldots,M$, or $\chi_\del=\chi^{k}_\del$, $k=1,\ldots,K$, satisfies
\begin{equation}\label{eq:chi_del_Linf_est}
	\del \|\nabla \chi_\del \|_{L^\infty(\Om)}
	\ \le\ C ,
\end{equation}
then there exists a constant $C_1$, independent of the coefficients $\hat u^{0,1},\ldots,\hat u^{0,M}$ and $\hat u^{1},\ldots,\hat u^{K}$ such that for every $\veps>0$ and $\del\in(0,\eta]$, the following estimates hold:
\begin{equation}\label{eq:est_BG}
	\|\nabla \vphi_0\|_{L^2(D_{\delta})}^2
	\ \leq\ C_1\max_m|\hat u^{0,m}|\, \varepsilon ,
\end{equation}
\begin{equation}\label{eq:est_eigs}
	\|\nabla \vphi_k\|_{L^2(D_\del)}^2 \ \le\  \frac{C_1}{\min_{j} |\hat u^j|}\, \veps ,
	\quad\text{and}\quad
	\lam_k \ \le\  \frac{C_1}{\min_{j} |\hat u^j|} ,
	\quad
	k=1,\ldots,K .
\end{equation}
Moreover, for each Lipschitz domain $V\sbt D_\eta\setminus A_\eta$ with $\cH^{d-1}(\p V\cap\p\Om)>0$, there exists a constant $C_2$, independent of the coefficients $\hat u^{0,1},\ldots,\hat u^{0,M}$ and $\hat u^{1},\ldots,\hat u^{K}$ such that for every $\veps>0$ and $\del\in(0,\eta]$, the following estimates are satisfied
\begin{equation}\label{eq:L2_est_BG}
	\|u^0-\vphi_0\|_{L^2(V)}^2 \le C_2\max_m|\hat u^{0,m}|\, \varepsilon ,
\end{equation}
\begin{equation}\label{eq:L2_est_eigs}
	\|\vphi_k\|_{L^2(V)}^2 \le \frac{C_2}{\min_{j} |\hat u^j|}\, \veps ,
	\qquad
	k=1,\ldots,K .
\end{equation}
\end{theorem}

\begin{proof}
Estimates \deq{eq:est_BG} and \deq{eq:est_eigs} easily follow from Lemma \ref{lem:SuffCondTVEst} with $p=\infty$, Lemmas \ref{lem:background_estimation} and~\ref{lem:eigVal_est}, and Proposition \ref{prop:eigFunc_est0}.
Estimates \deq{eq:L2_est_BG} and \deq{eq:L2_est_eigs} follow from the Poincar{\'e} inequality, and estimates \deq{eq:est_BG} and \deq{eq:est_eigs}, respectively, since $\cH^{d-1}(\p V \cap \p\Om)>0$.
\end{proof}

Essentially, estimates \deq{eq:est_BG} and \deq{eq:est_eigs} imply that $\vphi_0$ and the first $K$ eigenfunctions $\vphi_k$ of $L_\veps[u_\del]$ are ``almost'' constant in each connected component of $D_\del$.
In particular, $\vphi_0$ is almost constant in each connected component of $D_\del\setminus A_\del$.
Since the Hausdorff measure of $\p\Om\cap\p A^{0,m}$ is positive for each $m=1,\ldots,M$, and $\vphi_0$ coincides with $u^0_\del$ on $\p\Om$, $\vphi_0$ indeed approximates $u^0_\del$ well in $D_\del\setminus A_\del$.
Similarly, for $k=1,\ldots,K$, $\vphi_k$ is almost constant in each connected component of $D_\del\setminus A_\del$ and vanishes on $\p\Om$; therefore, every $\vphi_k$ is small in $D_\del\setminus A_\del$.
Consequently, $u_\del$ can be well approximated in $\vphi_0+\Phi_K$, for instance, by its $L^2$-best approximant $\vphi_0+\Pi_K(u_\del-\vphi_0)$, where $\Pi_K:\cV^\del\to\Phi_K$ denotes the standard $L^2$-orthogonal projection
\begin{equation}\label{eq:orthogonal-projector-PK}
	\angb{v-\Pi_Kv,\vphi}=0 ,
	\qquad
	\forall \, \vphi \in \Phi_K.
\end{equation}

From Theorem \ref{thm:main}, we may deduce estimates \deq{eq:est_BG}-\deq{eq:L2_est_eigs} for specific methods of approximation.

\begin{corollary}\label{cor:main}
Suppose $A^1,\ldots,A^K$ and $A^{0,1},\ldots,A^{0,M}$ have Lipschitz boundaries.
Estimates \deq{eq:est_BG}-\deq{eq:L2_est_eigs} hold true in each of the two following cases:
\begin{enumerate}
\item
For each $\del\in(0,\eta]$,  $u_\del$ is the convolution of $u$ with the standard mollifier (e.g., \cite{EG1992}), and $\cV^\del=H^1(\Om)$.

\item
For each $\del\in(0,\eta]$, $u_\del$ is the Lagrange interpolant of $u$ in an $H^1$-conforming FE space $V_\del$ associated with a simplex mesh $\cT_\del$ with mesh size $\del$, where the family of meshes $\{\cT_\del\}_{\del\in(0,\eta]}$ is regular and quasi-uniform (see, e.g., \cite{Q2008}), and either $\cV^\del=V_\del$ or $\cV^\del=H^1(\Om)$.
\end{enumerate}
\end{corollary}


\begin{proof}
For the proof, it is sufficient to show that in each case (i), (ii), the hypotheses of Theorem \ref{thm:main} are satisfied, i.e., that the approximation $\chi_\del$ of a characteristic function of a Lipschitz domain contained in $\Om$ satisfies \deq{eq:chi_del_Linf_est}.
Here we sketch the proof for (i); the result for (ii) follows easily from the definitions of a regular and quasi-uniform family of FE meshes and from basic properties of polynomial Lagrange interpolation.

Let $\chi$ be the characteristic function of a Lipschitz domain contained in $\Om$.
We extend $\chi$ to $\R^d$ by setting $\chi=0$ outside $\Om$, and set
\begin{equation}
	\chi_\del(x) = \zeta_\del*\chi = \int_{\R^d} \zeta_\del(x-y)\chi(y)\, d y ,
	\qquad
	\zeta_\del(x)=\del^{-d}\zeta(x/\del)
\end{equation}
with $\zeta$ the standard mollifier.
Since
\begin{equation}
	\nabla \chi_\del = \nabla \zeta_\del*\chi ,
\end{equation}
and $\|\chi\|_{L^\infty(\R^d)}=1$, we obtain
\begin{equation}
	\del |\nabla \chi_\del(x)| \le \|\chi\|_\infty\int_{|y|<\del} \del \, |\nabla \zeta_\del(y)|\, d y
		=\int_{|z|<1} |\nabla \zeta(z)|\, d z <\infty
\end{equation}
for a.e.\ $x\in\Om$, which concludes the proof.
\end{proof}

\begin{remark}\label{remark:MainThm}
Let us suppose, for simplicity, that $K=1$, $A^1=A$, $\hat u^1=1$, $u^0=0$, and the hypotheses of Theorem \ref{thm:main} are satisfied.
By following the proof for this simpler case, one finds that for small $\eta$ and $\veps$, $\lam_1$ is bounded from above by a constant arbitrarily close to
\begin{equation}\label{eq:Examp_LimitConst}
	\frac{\limsup_{\rho\to 0^+}\|\nabla\chi_\rho\|_{L^1(\Om)}}{\cL(A)} .
\end{equation}
For an appropriate approximation $\{\chi_\del\}_\del$ of the characteristic function $\chi^1$ of $A$, this constant essentially coincides with the eigenvalue
\begin{equation}\label{TV_EV}
	\lam=\frac{\cH^{d-1}(\p A)}{\cL(A)}
\end{equation}
of the (subdifferential of the) total variation (TV) functional associated with $\chi^1$ in two dimensions \cite{BCN2002}.
However, $\chi^1$ can be an eigenfunction of the TV functional associated with the eigenvalue \deq{TV_EV}, only if $A$ is convex and has boundary of class $C^{1,1}$.
In contrast, the conclusion of Theorem \ref{thm:main} is valid for any Lipschitz domain $A$; in particular, $A$ need not be convex.
Since the TV functional is only lower semi-continuous, we note that the quantities in \deq{eq:Examp_LimitConst} and \deq{TV_EV} may not be equal.

\end{remark}

\begin{remark}\label{rem:generalization}
The main results, given by Theorem \ref{thm:main} and Corollary \ref{cor:main}, essentially remain true if we replace the last inequality in \deq{eq:mu_cond} by the more general inequality \deq{eq:mu_cond_r} with $r\in[1,2]$.
The proofs only require slight modifications, as much of the analysis actually carries over from the case $r=1$.
The main adjustments are required in Lemmas \ref{lem:background_estimation} and \ref{lem:eigVal_est}, where estimates of $\mu_\veps[u_\del^0]|\nabla u_\del^0|$ and $\mu_\veps[\wt{u}_\del]|\nabla \wt{u}_\del|$ are employed.
Since the modifications for the two Lemmas are similar, we only address the former here.
Instead of the estimate
\begin{equation}
	\int_{S_\del} \mu_\veps[u^0_\del]\, |\nabla u^0_\del|^2\, dx
		\le \|\nabla u^0_\del\|_{L^1(\Om)}
\end{equation}
which relies on \deq{eq:mu_cond}, we obtain
\begin{equation}\label{eq:gen_B_est}
	\int_{S_\del} \mu_\veps[u^0_\del]\, |\nabla u^0_\del|^2\, dx
		\le \cL(S_\del)^{r-1}\|\nabla u^0_\del\|_{L^1(\Om)}^{2-r} 
\end{equation}
by using \deq{eq:mu_cond_r} and H\"older's inequality.
As a consequence, for $r\in(1,2]$, the conclusions of Theorem \ref{thm:main} and Corollary \ref{cor:main} essentially remain true, and can even be improved by introducing a small multiplicative term to the right hand sides of the inequalities.
\end{remark}

\subsection{Proofs}
\label{sec:MainProofs}

\begin{proof}[Proof of Lemma \ref{lem:SuffCondTVEst}]

\asserRef{lem:SuffCondTVEst_L}
We show that $\psi:[0,\eta]\to\R$ given by
\begin{equation}
	\psi\frb{\del}=\begin{cases}
			\del^{-1}\cL(U_\del) & \del\ne0 \\
			2\cH^{d-1}(\p A) & \del=0
		\end{cases}
\end{equation}
is bounded by proving that $\psi$ is continuous in $[0,\eta]$.
Since $A$ is a bounded Lipschitz domain, its boundary $\p A$ is $(d-1)$-rectifiable, i.e., there exists a Lipschitz function from a bounded subset of $\R^{d-1}$ onto $\p A$.
Hence the $(d-1)$-dimensional Minkowski content and the $(d-1)$-dimensional Hausdorff measure of $\p A$ coincide \cite[Theorem 3.2.39]{H1969}, that is
\begin{equation}\label{eq:MikowskiCont}
	\lim_{\del\to0^+}\del^{-1}\cL(U_\del)=2\cH^{d-1}(\p A) .
\end{equation}
Therefore $\lim_{\del\to0^+} \psi(\del)=\psi(0)$ and $\psi$ is continuous at $\del=0$.
Since $\rho(x)=\dist(x,\p A)$ is Lipschitz and satisfies $|\nabla\rho(x)|=1$, a.e.\ $x\in U_\del$ ~\cite[\S 6]{DZ2011}, we have
\begin{equation}
	\cL(U_\del)=\int_0^\del \cH^{d-1}\Brb{\rho^{-1}\brb{\{t\}}} \, d t ,
\end{equation}
by the co-area formula \cite[\S 3.4]{EG1992}.
Combining this with \deq{eq:MikowskiCont}, we conclude that the function $\del\mapsto\cL(U_\del)$ is continuous in $[0,\eta]$.
It follows that $\psi$ is continuous in $(0,\eta]$.
Since $\psi$ is also continuous at $\del=0$, we have that it is continuous in the entire closed interval $[0,\eta]$, and is, therefore, bounded.

\medskip
\asserRef{lem:SuffCondTVEst_g}
By H\"older's inequality we have
\begin{equation}
	\|g_\del\|_{L^1(\R^d)} =\int_{U_\del} |g_\del(x)|\, d x
		\le \Big(\cL(U_\del)\Big)^{1-1/p}\, \|g_\del\|_{L^p(\R^d)} ,
\end{equation}
which, together with \deq{eq:CondLpEst}, yields
\begin{equation}
	\|g_\del\|_{L^1(\R^d)} \le C_1 \Big(\del^{-1}\cL(U_\del)\Big)^{1-1/p} .
\end{equation}
Thus the conclusion follows from assertion \ref{lem:SuffCondTVEst_L}.
\end{proof}

\begin{proof}[Proof of Lemma \ref{lem:background_estimation}]
By Dirichlet's principle, $\vphi_0$ is the unique minimizer of the functional $I(v)=B[v,v]$ in $u_\del+\cV^\del_0$, i.e.,
\[
	B[\vphi_0,\vphi_0] \ \le\ B[u_\delta+v,u_\delta+v] ,
	\qquad
	\forall\, v\in\cV^\del_0 ,
\]
or, since $u^0_\delta-u_\delta=-\wt u_\del\in\cV^\del_0$, equivalently
\begin{equation}\label{eq:phi0_min}
	B[\vphi_0,\vphi_0] \le B[u^0_\delta+v,u^0_\delta+v] ,
	\qquad
	\forall\, v\in\cV^\del_0 .
\end{equation}
By \deq{eq:phi0_min} with $v=0$, and using \deq{eq:grad_u0_tildeu_supp}
\begin{equation}
	B[\vphi_0,\vphi_0] \leq B[\uBG_\del,\uBG_\del]
		=\int_{\Om} \mu_\veps[u_\del]|\nabla \uBG_\del|^2
		=\int_{S_\del} \mu_\veps[u_\del]|\nabla \uBG_\del|^2 .
\end{equation}
Then, using \deq{eq:mu_hatmu}, $\nabla u_\del=\nabla \uBG_\del$, which holds in $S_\del$, and \deq{eq:REstMu} we obtain
\begin{equation}
	B[\vphi_0,\vphi_0] \leq \int_{S_\del} \mu_\veps[u_\del]|\nabla \uBG_\del|^2
		=\int_{S_\del} \mu_\veps[\uBG_\del]|\nabla \uBG_\del|^2
		\leq \|\nabla \uBG_\delta\|_{L^1(S_\del)} .
\end{equation}
Finally, the conclusion follows from \deq{eq:grad_D_del_B_est}.
\end{proof}

The proof of Lemma \ref{lem:eigVal_est} requires the following result.

\begin{lemma}\label{lem:CharFuncLinInd}
If the sets $A^k_\eta$, $k=1,\ldots,K$, are nonempty and \deq{eq:USetsSep} is satisfied, for some $\eta>0$, then the following assertions hold true.
\begin{enumerate}
\item\label{lem:Aser1}
There exists $1\le j\le K$, such that
\begin{equation}\label{eq:lem:sepBound}
	\cB^j_\eta\setminus \bigcup_{\substack{k=1 \\ k\ne j}}^K \ol{\cB^k_\eta}\ne\emptyset ,
	\qquad
	\cB^\ell_\eta =A^\ell_\eta\cap D_\eta .
\end{equation}

\item\label{lem:Aser2}
The restrictions of the functions $\chi^k$, $k=1,\ldots,K$, to $D_\eta$ are linearly independent.\\

\item\label{lem:Aser3}
If for every $k=1,\ldots,K$, \deq{eq:approx_reg} is satisfied, then there exists a constant $C>0$, such that for every $\del\in[0,\eta]$ and $\what\psi=(\psi^k)\in\R^K$, the following estimates are satisfied:
\begin{equation}\label{eq:psi_delCoer}
	\|\psi_\del\|^2\ge \|\psi\|_{L^2(D_\eta)}^2 \ge C|\what\psi|^2 ,
\end{equation}
where $\psi=\sum_k\psi^k \chi^k$, and $\psi_\del=\sum_k\psi^k \chi^k_\del$.
\end{enumerate}
\end{lemma}

\begin{proof}
\asserRef{lem:Aser1}
Since one set  $A^k_\eta$ may be included inside another, the index $j$ must be such that the boundary of $A^j_\eta$ is contained in the boundary of the union $A_\eta=\cup_k A^k_\eta$.
We show that this requirement is sufficient:
Fix some $x\in \p A_\eta$.
Then, for each neighborhood $W$ of $x$, there holds
\begin{equation}
	W\cap A_\eta\ne\emptyset ,
	\quad
	W\cap (\Om\setminus A_\eta)\ne\emptyset .
\end{equation}
Since
\begin{equation}
	W\cap A_\eta =\bigcup_{k=1}^K \Brb{W\cap A^k_\eta} ,
\end{equation}
and
\begin{equation}
	W\cap (\Om\setminus A_\eta)
		=\bigcap_{k=1}^K \Brb{W\cap (\Om\setminus A^k_\eta) } ,
\end{equation}
we have $x\in\p A^j_\eta$, for some $1\le j \le K$.
Since $\p A^j_\eta\sbt\p U^j_\eta$ and, by \deq{eq:USetsSep}, the sets $\ol{U^k_\eta}$ are mutually disjoint, $x\notin \p A^k_\eta$, for all $k\ne j$.
Similarly, we have $x\notin\p S_\eta$.
Hence, there exists an open neighborhood $W$ of $x$, $W\cap (\ol{A^k_\eta\cup U^k_\eta})=\emptyset$, for all $k\ne j$, and $W\cap\ol{S_\eta}=\emptyset$.
Moreover, since $x\in \p A^j_\eta$, the intersection $V=W\cap A^j_\eta$ is nonempty,
while $V\cap \ol{U^k_\eta}=\emptyset$, for all $k=1,\ldots,K$, and $V\cap \ol{S_\eta}=\emptyset$.
Thus, $V$ is nonempty, open and satisfies $V\sbt A^j_\eta\cap D_\eta=\cB^j_\eta$ and $V\cap \ol{\cB^k_\eta}\sbt V\cap \ol{A^k_\eta}=\emptyset$, for $k\ne j$, by \deq{eq:setA} and \deq{eq:USetsSep}.
This yields \deq{eq:lem:sepBound}.

\medskip
\asserRef{lem:Aser2}
Suppose
\begin{equation}\label{eq:linCombZero}
	\sum_{k=1}^K \psi^k \chi^k=0
	\qquad
	\text{a.e.\ in $D_\eta$}
\end{equation}
for some $\psi^k\in\R$, $k=1,\ldots,K$.
According to \ref{lem:Aser1}, there is an index $j$ for which \deq{eq:lem:sepBound} holds.
Note that the set on the left-hand side of \deq{eq:lem:sepBound} is open and therefore is of positive measure.
Without loss of generality, assume that $j=K$ satisfies \deq{eq:lem:sepBound}.
However, by \deq{eq:linCombZero} this can be, only if $\psi^K=0$.
The above argument may be repeated by induction.
Since there is a finite number of functions $\chi^k$, the procedure stops only when a single function remains, say $\chi^1$.
Then \deq{eq:linCombZero} reduces to $\psi^1\chi^1=0$.
Again, since $A^1_\eta\cap D_\eta$ is open and nonempty, this implies $\psi^1=0$ and the restrictions of $\chi^k$ to $D_\eta$ are thus linearly independent.

\medskip
\asserRef{lem:Aser3}
Let $\what\psi=(\psi^k)\in\R^K$, $\del\in[0,\eta]$, $\psi=\sum_k\psi^k \chi^k$, and $\psi_\del=\sum_k\psi^k \chi^k_\del$, where $\chi^k_0=\chi^k$.
As each $\chi^k_\del$ coincides with $\chi^k$ in $D_\del$, we have
\begin{equation}
	\|\psi_\del\|^2 \ge \int_{D_\del} \Big|\sum_{k=1}^K \psi^k\chi^k_\del \Big|^2
		=\int_{D_\del} \Big|\sum_{k=1}^K \psi^k\chi^k \Big|^2
		=\|\psi\|^2_{L^2(D_\del)} .
\end{equation}
Since $D_\eta\sbt D_\del$, $\del\le\eta$, we obtain the lower bound
\begin{equation}
	\|\psi_\del\|^2 \ge \|\psi\|^2_{L^2(D_\del)} \ge \|\psi\|^2_{L^2(D_\eta)} .
\end{equation}
Expanding the term on the right-hand side yields
\begin{equation}
	\|\psi\|^2_{L^2(D_\eta)} = \sum_{k,j=1}^K \psi^k\psi^j \int_{D_\eta} \chi^k\chi^j
		=\sum_{k,j=1}^K M_{kj} \psi^k\psi^j
\end{equation}
where
\begin{equation}
	M_{kj}=\int_{D_\eta} \chi^k\chi^j =\cL(A^k\cap A^j\cap D_\eta)
	\qquad
	k,j=1,\ldots,K .
\end{equation}
Hence $\|\psi\|^2_{L^2(D_\eta)}$ is a quadratic form in the vector $\what\psi$, represented by the symmetric $K\times K$ matrix $M=(M_{kj})$, and it is positive for any $\what\psi\ne0$ by assertion \ref{lem:Aser2} of this lemma.
Therefore, $M$ is symmetric positive definite, which implies \deq{eq:psi_delCoer} and hence the conclusion.
\end{proof}

\begin{proof}[Proof of Lemma \ref{lem:eigVal_est}]
From the spectral theory for symmetric elliptic operators~\cite[\S 6.5 -- 6.6]{E2010}, it follows that for each $n\ge 1$,
\begin{equation}\label{eq:rayleigh_quotient}
	\lam_n=\min_{0\ne v\in \Phi_{n-1}^\bot} \frac{B[v,v]}{\|v\|^2} ,
\end{equation}
where $\Phi_{n-1}=\Span\{\vphi_j\}_{j=1}^{n-1}$ for $n\ge2$, or $\Phi_0=\{0\}$ for $n=1$, and $W^\bot$ denotes the orthogonal complement of $W$ in $\cV^\del_0$ with respect to the $L^2(\Om)$ inner product.
For any
\begin{equation}\label{eq:psi_del_def}
	\psi_\del=\sum_{k=1}^K \psi^k \chi^k_\del ,
	\qquad
	\psi^k\in\R ,
\end{equation}
there holds
\begin{equation}
	\angb{\vphi_j, \psi_\del} =\sum_{k=1}^K \psi^k\angb{\vphi_j, \chi^k_\del} .
\end{equation}
Now, let $1\le n\le K$ and $\what\psi=(\psi^k)\in\R^K$, with $|\what\psi|=1$, such that $\psi_\del$ given by \deq{eq:psi_del_def} satisfies
\begin{equation}\label{eq:psiInPhiOC}
	\angb{\vphi_j, \psi_\del}=0 ,
	\qquad
	j=1,\ldots,n-1 ,
\end{equation}
for $n\ge 2$, or no condition at all for $n=1$.
One can always find such a vector of coefficients $\what\psi$, since $n\le K$ and hence the homogeneous linear system \deq{eq:psiInPhiOC} has more unknowns than equations.
By Lemma \ref{lem:CharFuncLinInd} there exists a constant $C_1>0$, independent of $|\what\psi|=1$, $\del\in(0,\eta]$ or $\veps$, such that $\|\psi_\del\|^2\ge C_1$.
Therefore, $\psi_\del\in\Phi_{n-1}^\bot$, is not identically zero, and we obtain from \deq{eq:rayleigh_quotient}
\begin{equation}\label{eq:proof_lam_n_est}
	\lam_n\le \frac{B[\psi_\del,\psi_\del]}{\|\psi_\del\|^2} \le \frac{B[\psi_\del,\psi_\del]}{C_1} .
\end{equation}

Thus, it is left to estimate $B[\psi_\del,\psi_\del]$.
By \deq{eq:approx_reg}, we have
\begin{equation}
	B[\psi_\del,\psi_\del]
		=\int_{\Om} \mu_\veps[u_\del] |\nabla \psi_\del|^2 \, d x
		=\int_{U_\del} \mu_\veps[u_\del] |\nabla \psi_\del|^2 \, d x .
\end{equation}
Furthermore, by employing \deq{eq:mu_hatmu} and the equality $\nabla u_\del=\nabla \wt{u}_\del$ which holds a.e.\ in $U_\del$ because $S_\del\cap U_\del=\emptyset$, we get
\begin{equation}\label{eq:B_psi}
	B[\psi_\del,\psi_\del]
		=\int_{U_\del} \mu_\veps[\wt u_\del] |\nabla \psi_\del|^2 \, d x .
\end{equation}

Next, we shall show that
\begin{equation}\label{eq:est_psi_mu}
	|\nabla \psi_\del\frb{x}| < \frac{|\nabla \wt u_\del\frb{x}|}{m} ,
	\qquad
	\text{a.e.\ $x\in U_\del$,}
\end{equation}
with $m=\min_j |u^j|>0$.
Since the sets $U^1_\del,\ldots,U^K_\del$ are mutually disjoint, each $x\in U_\del$ lies in precisely one $U^k_\del$, with $1\le k\le K$, and hence at a.e.\ $x\in U_\del$, at most one of the gradients, $\nabla \chi^k_\del(x)$, is nonzero.
Therefore, a.e.\ in $U_\del$ we have
\begin{equation}
\begin{aligned}
	|\nabla \psi_\del|
        &=\Big|\sum_{k=1}^K \psi^k\nabla \chi^k_\del \Big|
			= \sum_{k=1}^K |\psi^k||\nabla \chi^k_\del | \\
		&\le \frac{1}{m}\, \sum_{k=1}^K |u^k| |\nabla \chi^k_\del |
			= \frac{1}{m}\, \abs{\sum_{k=1}^K u^k \nabla \chi^k_\del }
\end{aligned}
\end{equation}
which completes the proof of \deq{eq:est_psi_mu}.

Finally, from \deq{eq:B_psi}, \deq{eq:est_psi_mu} and \deq{eq:REstMu} with $w=u_\del$, we infer that
\begin{equation}
	B[\psi_\del,\psi_\del]
		\le \frac{1}{m} \int_{U_\del} |\nabla \psi_\del|\, d x
		=\frac{1}{m}\, \|\nabla\psi_\del\|_{L^1(\Om)} .
\end{equation}
By using the definition of $\tau_k$ in \deq{eq:ThmEst} and applying the Cauchy-Schwarz inequality in $\R^K$ we obtain
\begin{equation}
	B[\psi_\del,\psi_\del]\le \frac{1}{\min_j |u^j|}\, \sum_{k=1}^K |\psi^k| \tau_k
		\le\frac{|\tau|}{\min_j |u^j|} ,
\end{equation}
which, together with \deq{eq:proof_lam_n_est}, yields the conclusion.
\end{proof}

\subsection{Numerical example}\label{sbSec:Numerical example}

\begin{figure}[t]
    \centering
    \subcaptionbox{true medium $u$ (or $u_{\del}$)\label{fig:example-1-approximation-exact-profile}}{%
	\includegraphics[width=0.35\linewidth]{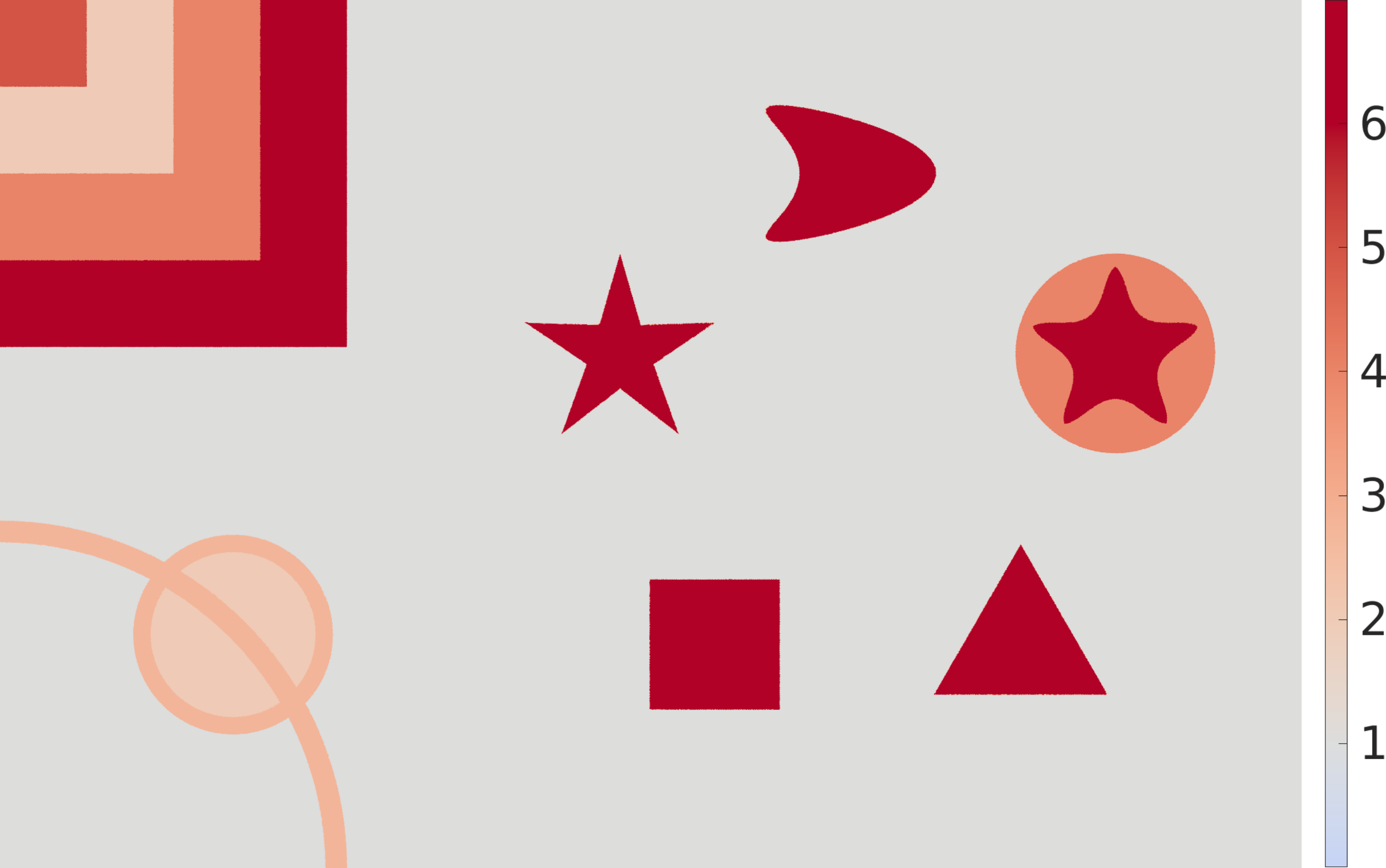}}
    \hspace{0.5cm}
    \subcaptionbox{AS approximation\label{fig:example-1-approximation-AE-approximation}}{%
        \includegraphics[width=0.35\linewidth]{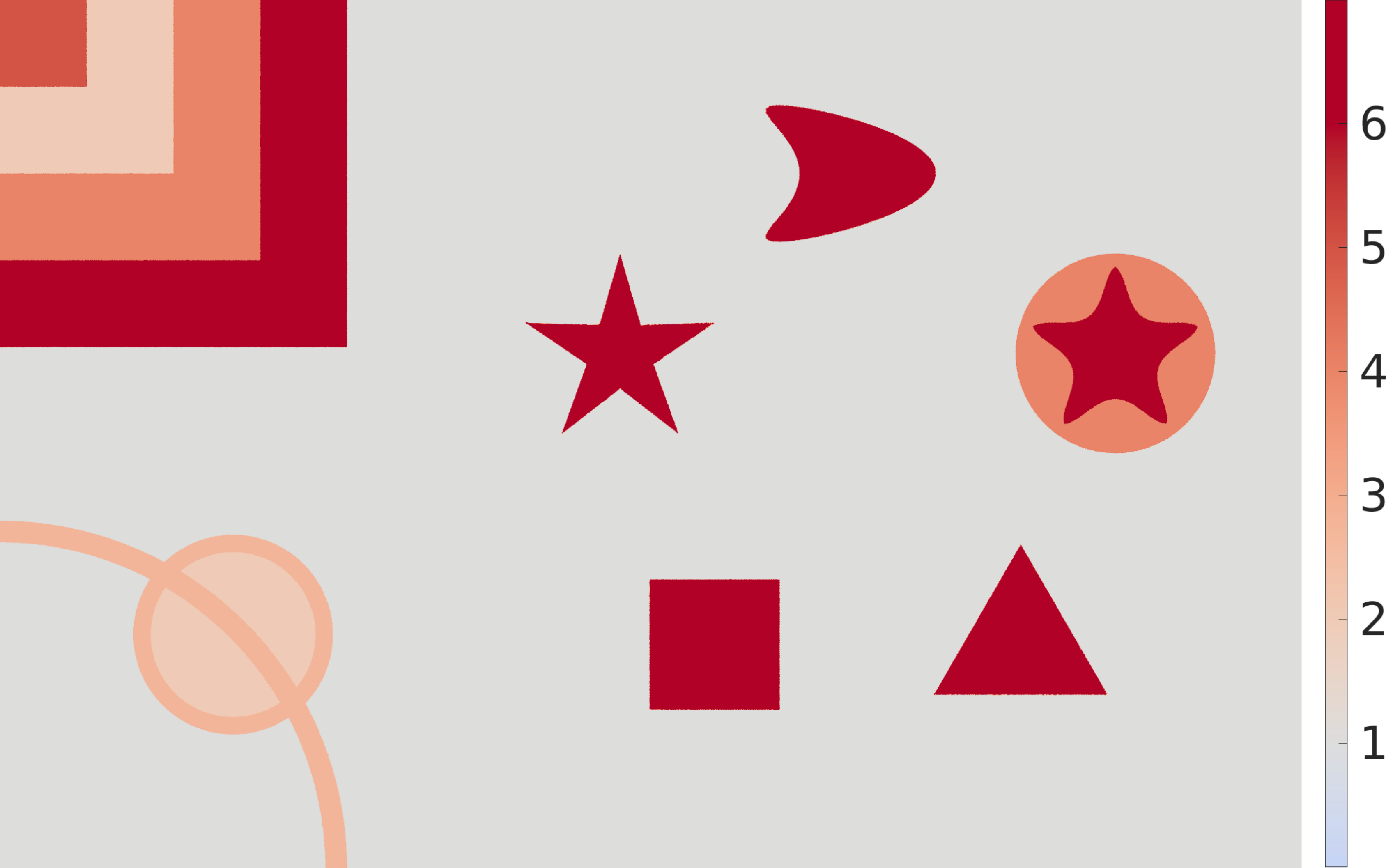}}

    \caption{Two-dimensional piecewise constant medium: (a) true medium $u$ (or $u_{\del}$),
        (b) AS approximation $\varphi_0 + \Pi_K[u_\delta-\vphi_0]$.
        Note that $u$ and its FE-interpolant $u_{\del}$ cannot be distinguished here.
    }
    \label{fig:example-1-approximation}
\end{figure}
\begin{figure}[t]
    \centering
    \def\expASlambdaA{5.9}
    \def\expASlambdaB{6.0}
    \def\expASlambdaC{7.1}
    \def\expASlambdaD{7.8}
    \def\expASlambdaE{13.7}
    \def\expASlambdaF{43.2}
    \def\expASlambdaG{55.4}
    \def\expASlambdaH{62.3}

    \subcaptionbox{$\vphi_0$\label{fig:example-1-background}}{%
	\includegraphics[width=0.3\linewidth]{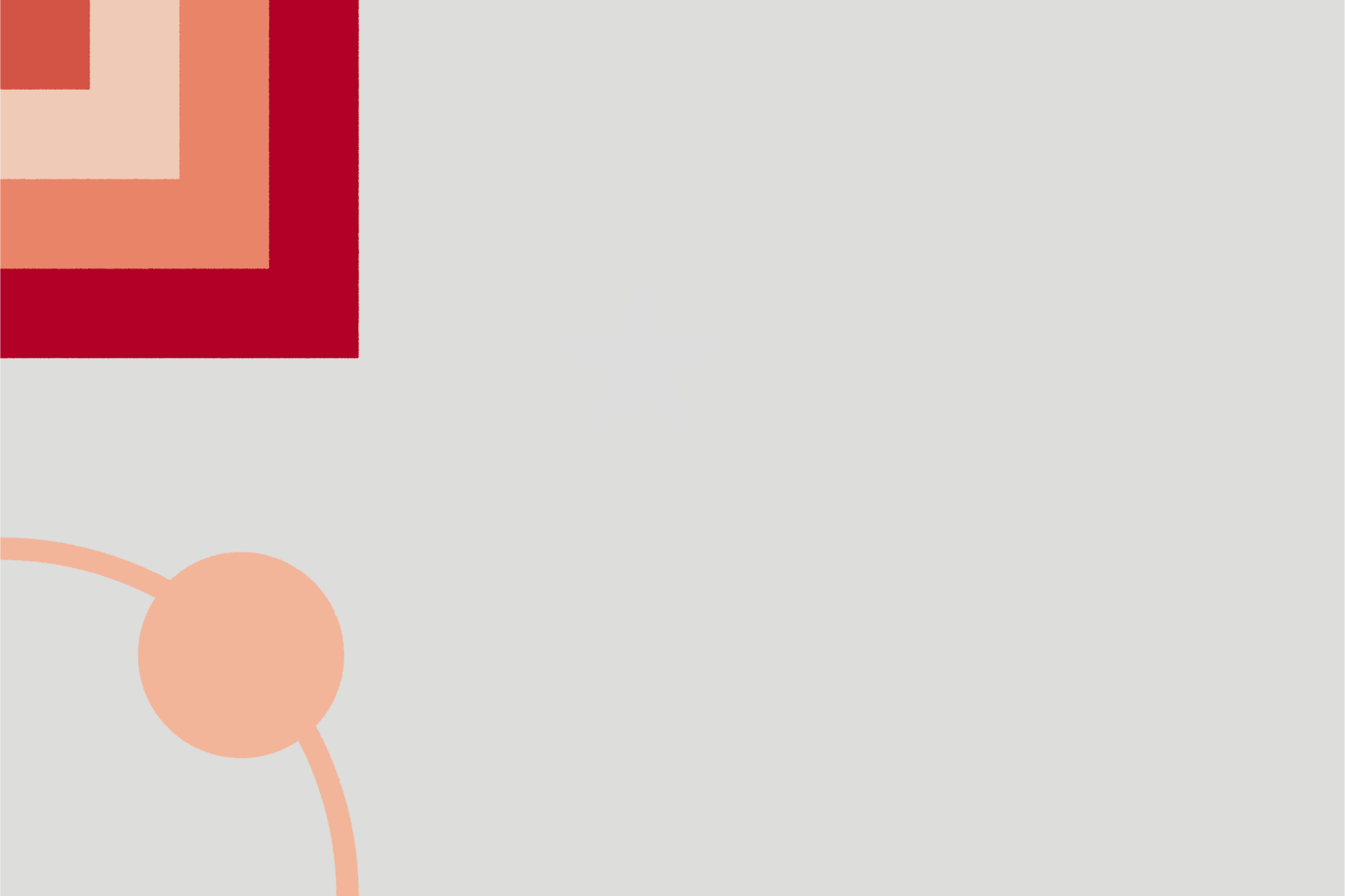}}
    \hspace{0.25cm}
    \subcaptionbox{$\vphi_1$ with  $\lambda_1 \approx \expASlambdaA$\label{fig:example-1-eigenfunction-1}}{%
        \includegraphics[width=0.3\linewidth]{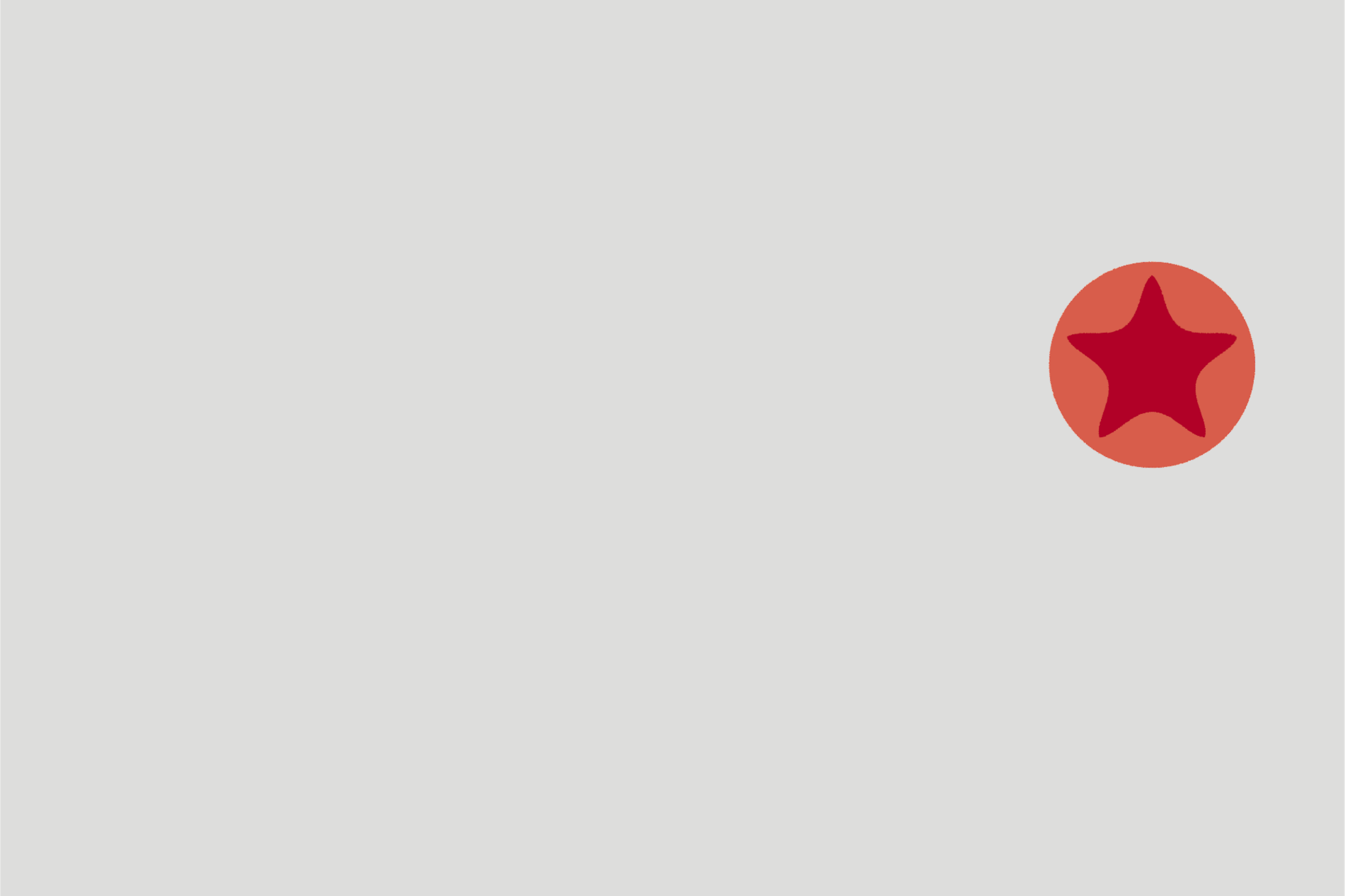}}
    \hspace{0.25cm}
    \subcaptionbox{$\vphi_2$ with  $\lambda_2 \approx \expASlambdaB$\label{fig:example-1-eigenfunction-2}}{%
        \includegraphics[width=0.3\linewidth]{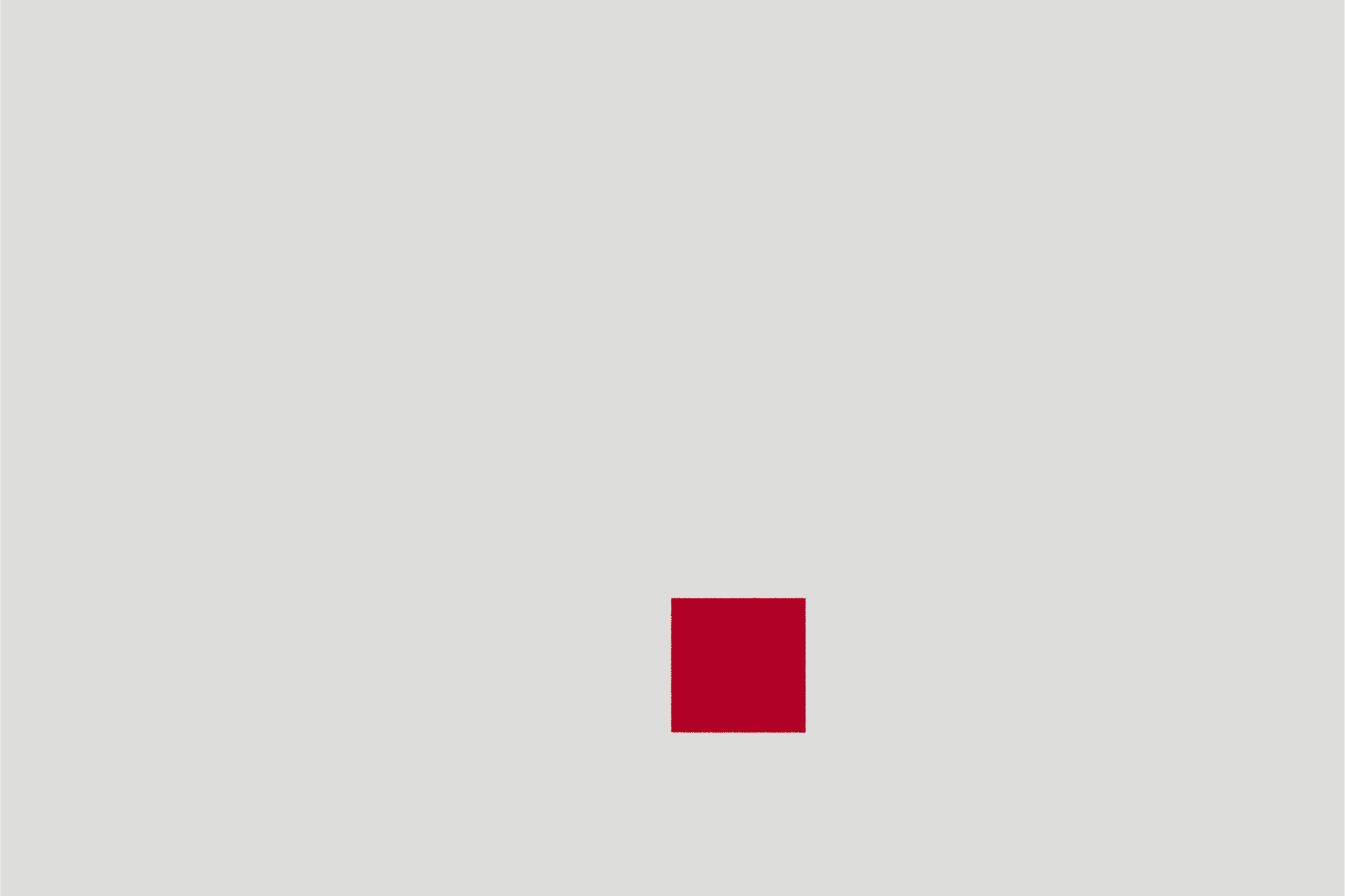}}
    \\[2ex]

    \subcaptionbox{$\vphi_3$ with  $\lambda_3 \approx \expASlambdaC$\label{fig:example-1-eigenfunction-3}}{%
        \includegraphics[width=0.3\linewidth]{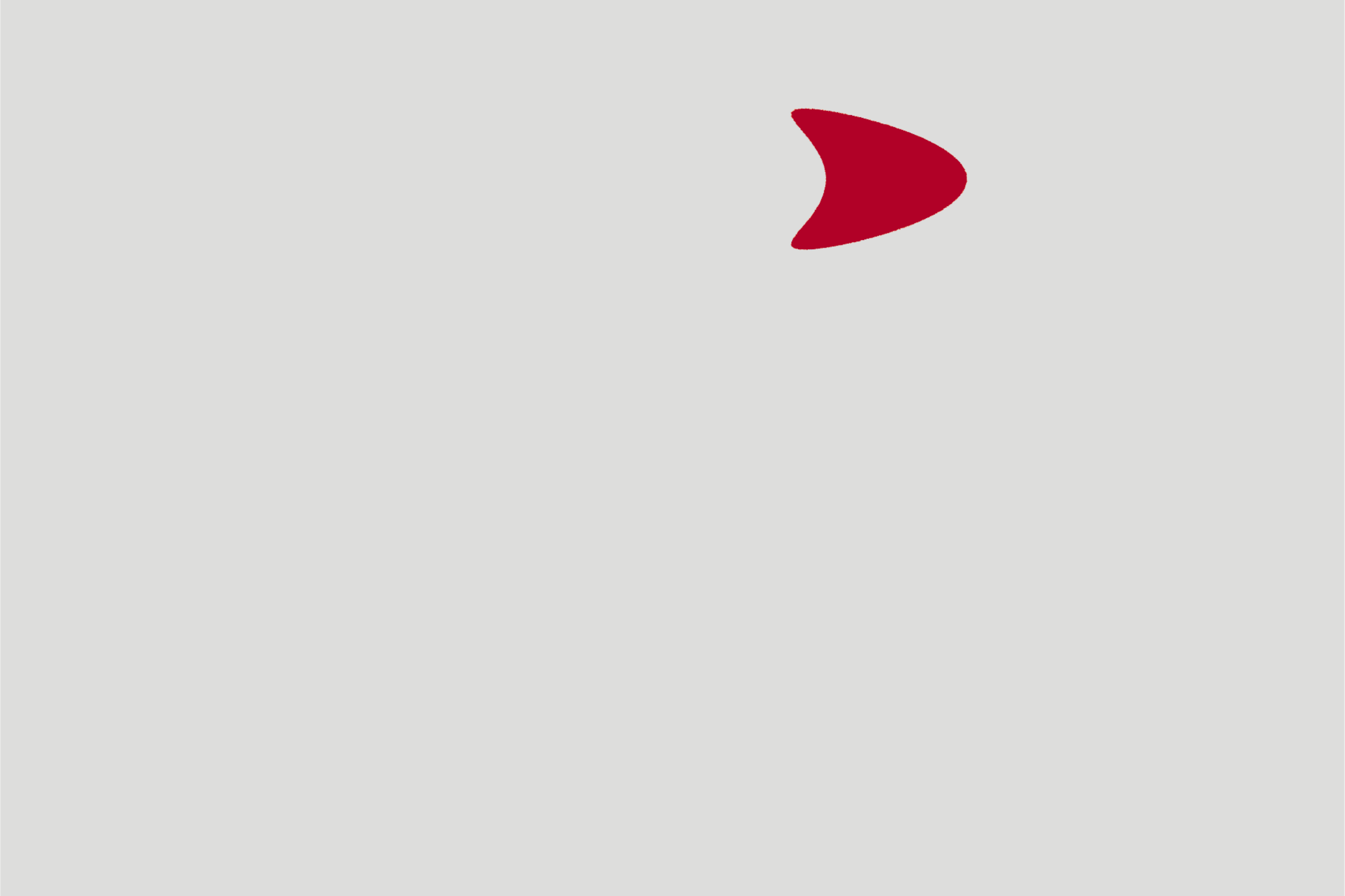}}
    \hspace{0.25cm}
    \subcaptionbox{$\vphi_4$ with  $\lambda_4 \approx \expASlambdaD$\label{fig:example-1-eigenfunction-4}}{%
        \includegraphics[width=0.3\linewidth]{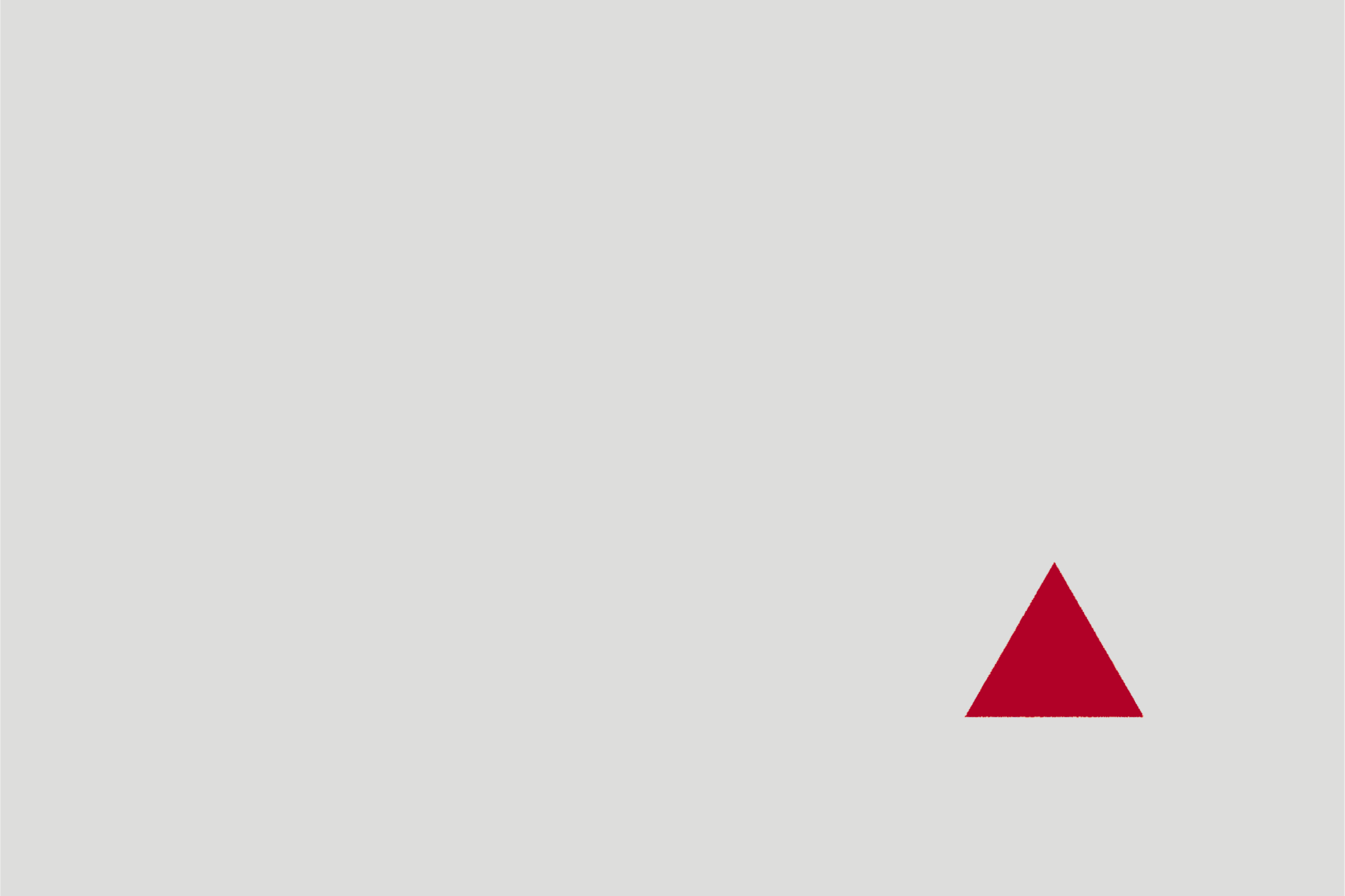}}
    \hspace{0.25cm}
    \subcaptionbox{$\vphi_5$ with  $\lambda_5 \approx \expASlambdaE$\label{fig:example-1-eigenfunction-5}}{%
        \includegraphics[width=0.3\linewidth]{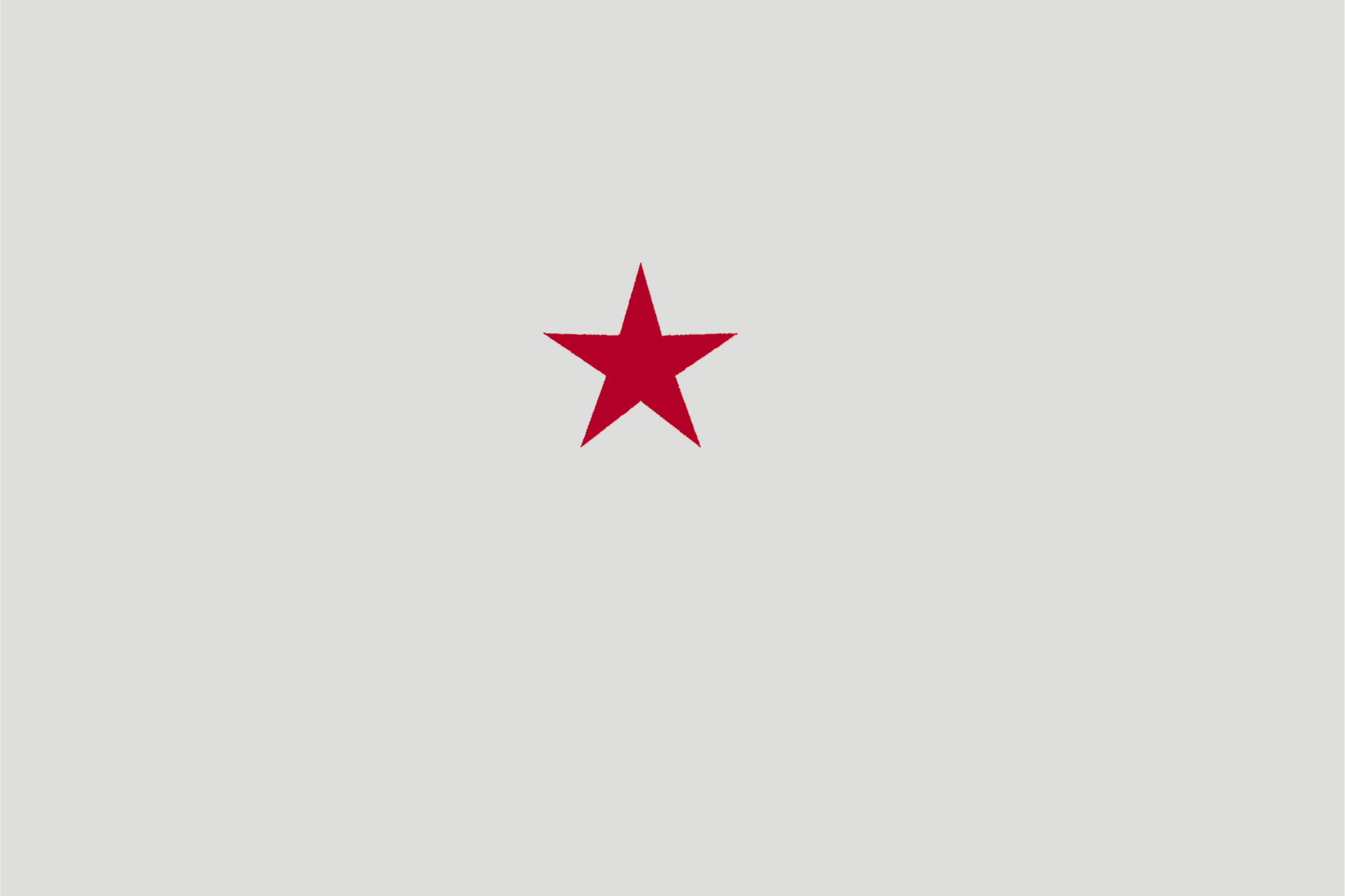}}
    \\[2ex]

    \subcaptionbox{$\vphi_6$ with  $\lambda_6 \approx \expASlambdaF$\label{fig:example-1-eigenfunction-6}}{%
        \includegraphics[width=0.3\linewidth]{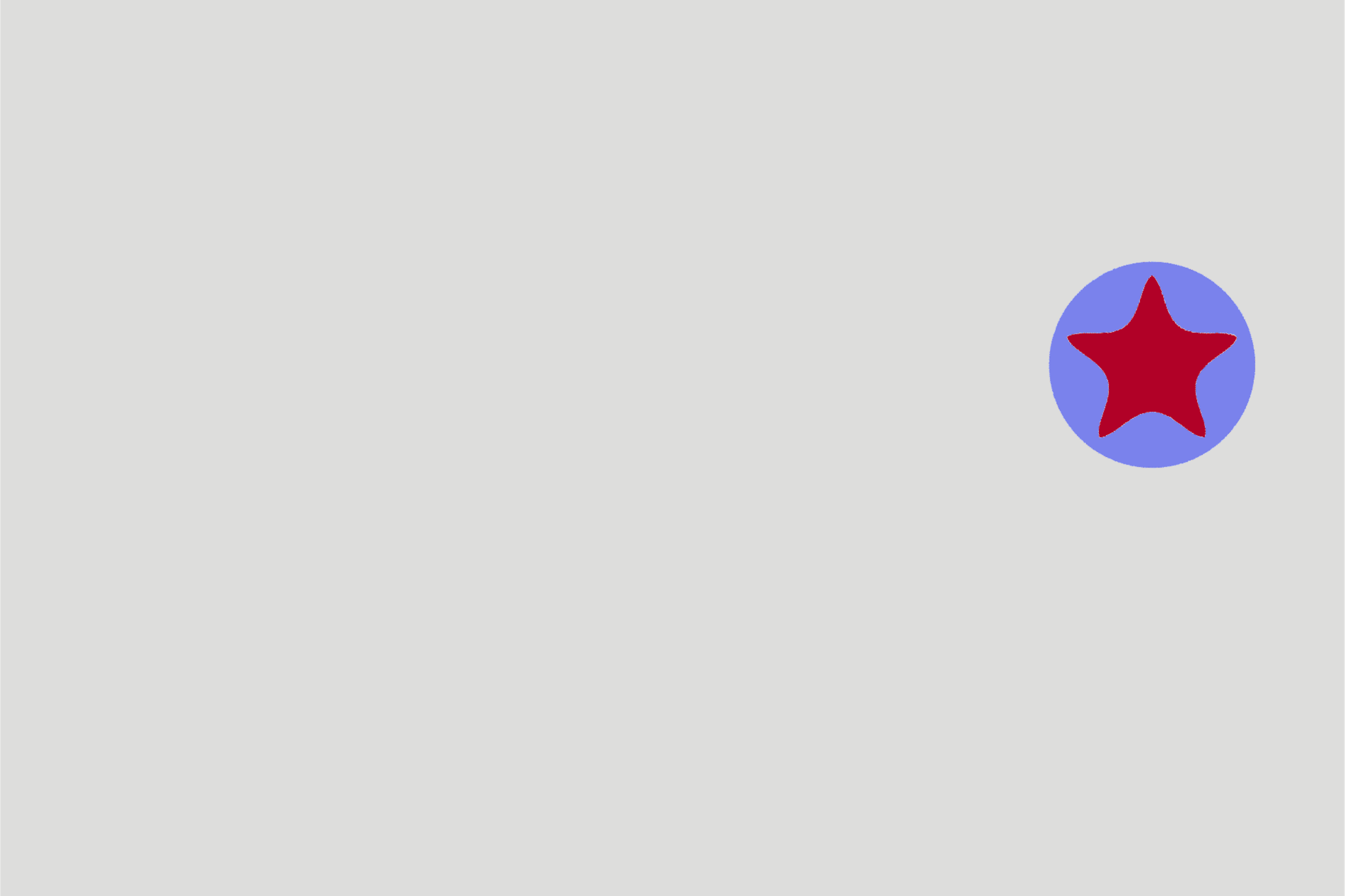}}
    \hspace{0.25cm}
    \subcaptionbox{$\vphi_7$ with  $\lambda_7 \approx \expASlambdaG$\label{fig:example-1-eigenfunction-7}}{%
        \includegraphics[width=0.3\linewidth]{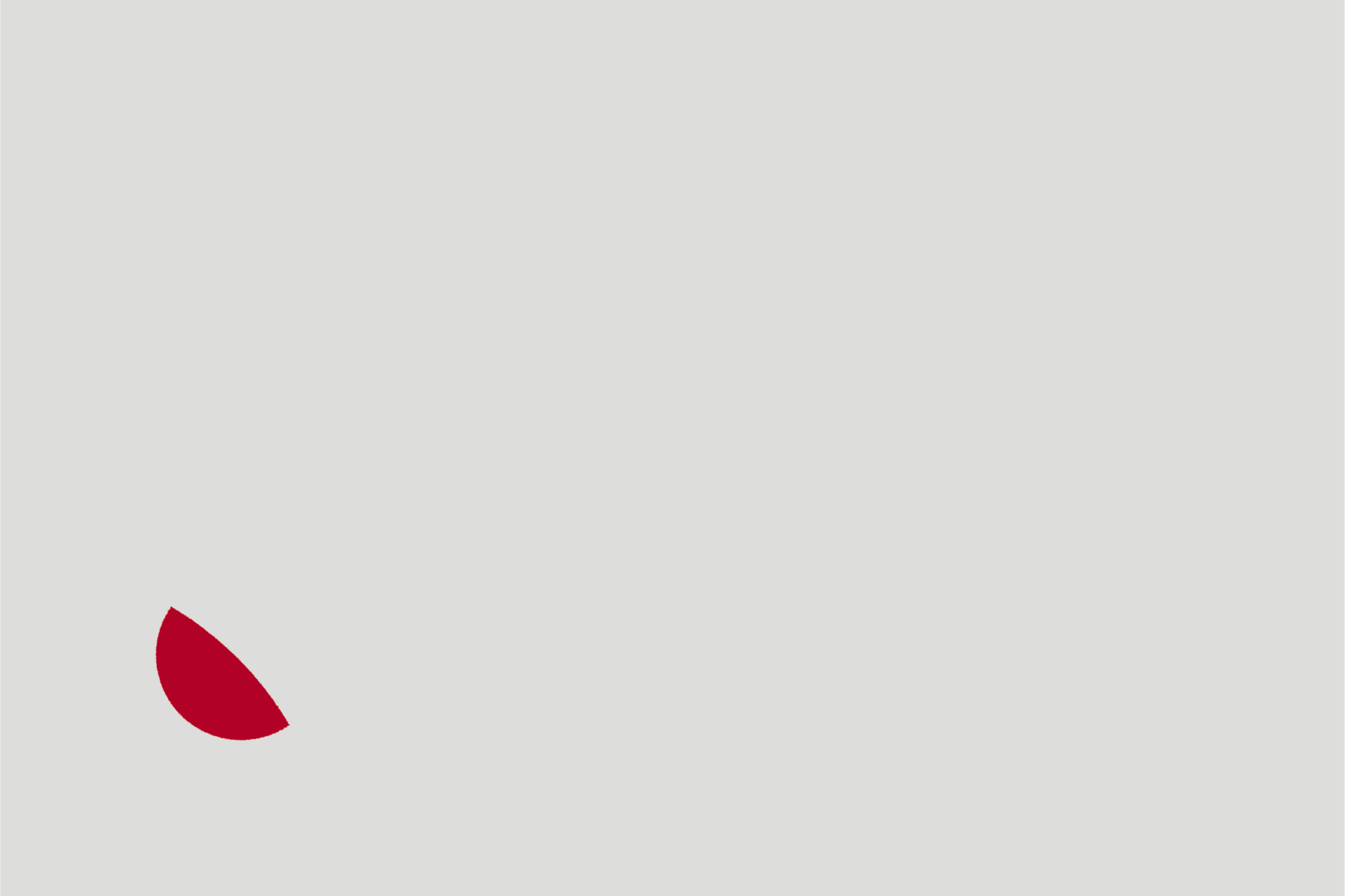}}
    \hspace{0.25cm}
    \subcaptionbox{$\vphi_8$ with  $\lambda_8 \approx \expASlambdaH$\label{fig:example-1-eigenfunction-8}}{%
        \includegraphics[width=0.3\linewidth]{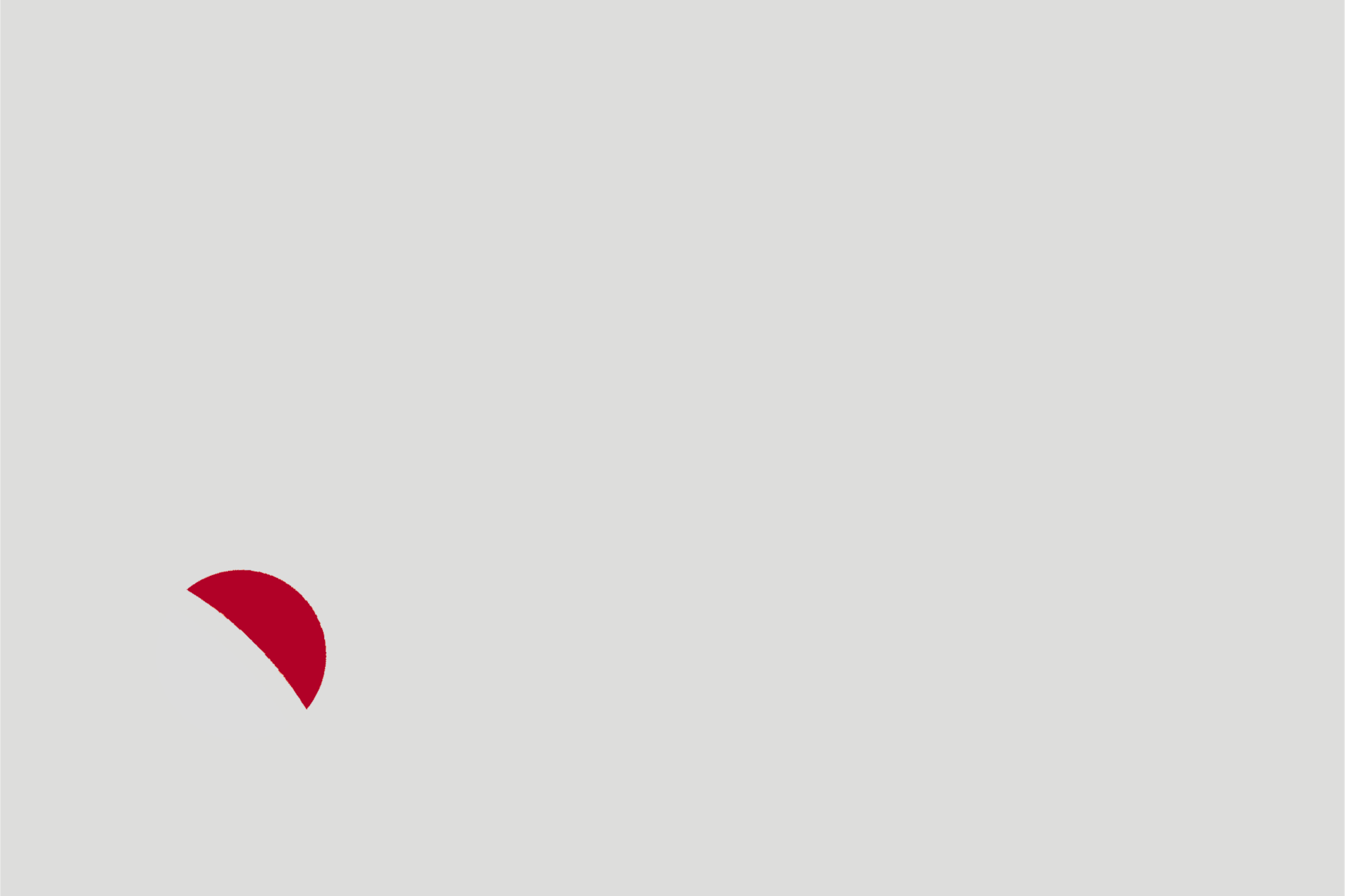}}

    \caption{
        Two-dimensional piecewise constant medium:
        (a) background $\vphi_{0}$ in (\ref{eq:phi0BVP_analysis}) with $w=u_{\del}$ and $\varepsilon=10^{-8}$;
        (b)-(i) eigenfunctions $\vphi_{1},\ldots,\vphi_{8}$ of \deq{eq:eigenValProb_analysis} corresponding to
        the first eight eigenvalues $\lambda_1,\ldots,\lambda_8$.
    }
    \label{fig:example-1-eigenfunctions}
\end{figure}

\begin{figure}
    \centering
    \includegraphics[width=12cm]{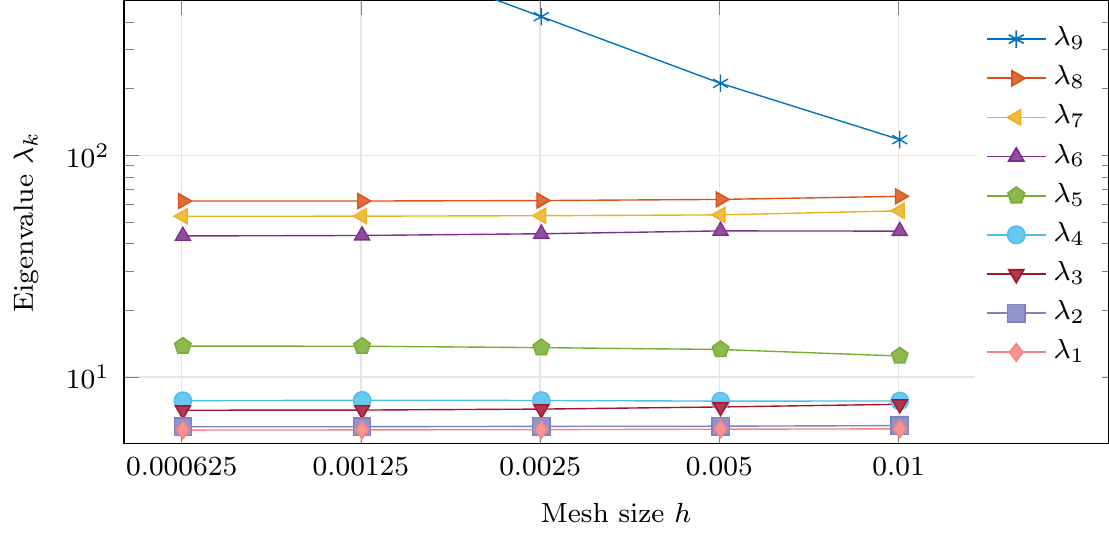}
    \caption{Two-dimensional piecewise constant medium:
    eigenvalues $\lambda_1,\ldots, \lambda_9$ of $L_\veps[u_{\del}]$ for mesh sizes $\delta = h=0.01/2^{m}$, $m=0,\ldots,4$, and fixed $\veps=10^{-8}$.
    }
     \label{fig:example-1-eigenvalues}
\end{figure}

%
To illustrate the remarkable approximation properties of
the adaptive spectral (AS) representation in Theorem \ref{thm:main} and Corollary \ref{cor:main},
we now consider in $\Omega = (0,1.5) \times (0,1)$ the piecewise constant medium $\uTotal = \uBG  + \uMedium$,
shown in Fig.\ \ref{fig:example-1-approximation-exact-profile}.
It consists of a background, $ \uBG $, and an interior part, $\uMedium$,
which vanishes on the boundary $\partial \Omega$.
Both $\uBG$ and $\uMedium$ are linear combinations,
\[
    \uBG(x)        =  \sum_{m=1}^{M} \hat u^{0,m} \chi^{0,m}(x),
    \qquad
    \uMedium(x)    =   \sum_{k=1}^{K} \hat u^{k} \chi^k(x),
    \qquad \hat u^{0,m}, \hat u^{k} \in \R
\]
of characteristic functions $\chi^{0,m}$ and $\chi^{k}$ associated with given subsets $A^{0,m}$ and $A^{k}$ of $\Omega$, respectively, with $M=7$ and $K=8$.
While the subsets $A^{0,m}$ composing the background reach the boundary  $\partial \Omega$ in the sense of $\mathcal{H}^{d-1}(\p A^{0,m}\cap \partial \Omega)>0$, where $\mathcal{H}^{d-1}$ denotes the $(d-1)$-dimensional Hausdorff measure, the inclusions $A^k$ lie strictly inside $\Omega$.

First, we approximate $\uTotal$ by its $H^1$-conforming $\mathcal{P}^1$-FE interpolant, $u_{\del} \in W^{1,\infty}(\Omega)$,
on a regular triangular mesh, locally refined along discontinuities, with
$242'790$ elements varying in size between $4.4\times 10^{-4}$ and $3.8\times 10^{-1}$.
Due to the fine adaptively constructed mesh, the FE- approximation errors are essentially negligible.
In fact, since $\uTotal$ and $u_{\del}$ differ by a relative $L^2$-error as little as $2.5\%$, they are hardly distinguishable in Fig.\ \ref{fig:example-1-approximation-exact-profile}.

Now, we compute the approximation $\vphi_0$ of the background by solving (\ref{eq:phi0BVP_analysis}) with
$L_\veps[u_{\del}]$ and $\mu_{\veps}[u_{\del}]$ as in (\ref{eq:linear_op}), (\ref{eq:mutwo}) and $\varepsilon=10^{-8}$.
As shown in Fig.\ \ref{fig:example-1-background}, $\vphi_0$ appears essentially piecewise constant
throughout $\Omega$ while correctly representing the various components of the background $\uBG$.
Similarly, the first eight eigenfunctions $\vphi_1,\ldots,\vphi_8$ of $L_\veps[u_{\del}]$,
defined by (\ref{eq:eigenValProb_analysis}), correctly identify in
Figs.\ \ref{fig:example-1-eigenfunction-1} -- \ref{fig:example-1-eigenfunction-8}
all remaining interior sets (or inclusions). For all isolated inclusions, each eigenfunction $\vphi_k$
accurately matches a single characteristic function $\chi^k$.
However, for the two overlapping star- and disk-shaped sets, ${A^1}$ and ${A^6}$, each of the remaining two eigenfunctions $\vphi_1$ and $\vphi_6$ essentially correspond to a linear combination of $\chi^1$ and $\chi^6$; hence, $\vphi_1$ and $\vphi_6$ span the same two-dimensional subspace as $\chi^1$ and $\chi^6$.
Note that neither $\vphi_0$ nor
$\vphi_1,\ldots,\vphi_8$ are truly piecewise constant but, in fact, lie in $H^1(\Om)$.
Although their gradients nearly vanish in $D_\delta$ (i.e., outside the $\delta$ neighborhoods of the interfaces), due to the upper bounds (\ref{eq:est_BG}) and (\ref{eq:est_eigs}) of Theorem \ref{thm:main}, the eigenfunctions do vary (slightly) throughout $\Omega$.

Next, we subtract the background $\vphi_0$ from $u_{\del}$ and compute
the $L^2$-projection $\Pi_K[u_\delta-\vphi_0]$ into $\Phi_K = \Span\{\varphi_i\}_{i = 1}^K$,
defined in (\ref{eq:orthogonal-projector-PK}). Since $\vphi_0$ matches well $\uBG$ and
$\vphi_1,\ldots,\vphi_8$ essentially span the same eight-dimensional subspace as
$\chi_1, \ldots, \chi_8$, $\Pi_K[u_\delta-\vphi_0]$ approximates $\uMedium$ remarkably well.
Hence, the combined AS representation $\varphi_0 + \Pi_K[u_\delta-\vphi_0]$, shown in Figure
\ref{fig:example-1-approximation-AE-approximation}, also approximates remarkably
well the entire medium $\uTotal$, or $u_{\del}$, with a relative $L^2$-error as little as $2.5\%$, or $0.05\%$, respectively.
In fact, $\uTotal$ is hardly
distinguishable from the AS approximation $\varphi_0 + \Pi_K[u_\delta-\vphi_0]$
 in  Figure~\ref{fig:example-1-approximation}.

Finally, we monitor in Fig.\ \ref{fig:example-1-eigenvalues} the behavior of the first nine eigenvalues of
$L_\veps[u_{\del}]$ for a sequence of increasingly finer quasi-uniform FE meshes and fixed $\veps=10^{-8}$.
While the first eight eigenvalues remain bounded, thereby validating the upper bound  (\ref{eq:est_eigs}),
the ninth eigenvalue apparently diverges as $h$ tends to zero.

%
%
%
%
%
%
%
%
%
%
%
%
%
%
%
%
\section{Numerical experiments}
\label{sec:NumRes}

Here we present two numerical examples which illustrate the accuracy and the usefulness of the ASI method
for the solution of inverse medium problems.
In the first, the unknown medium is composed of five simple geometric inclusions, and in the second, the medium is a two-dimensional model of a salt dome from geophysics.

We apply the ASI Algorithm from Section 2.2 with frequency stepping for the minimization of the misfit $\cJ$, given by (\ref{eq:cost_functional}), where the forward solution operator of the boundary value problem is replaced by a $\mathcal{P}^3$-FE Galerkin approximation.
For  $w$ a continuous, piecewise linear FE function, the operator $L_\veps[w]$, needed in Step 5 of the ASI Algorithm, is given by \deq{eq:linear_op}, where, $\mu_\veps = \mu_\veps[w]$ is given by \deq{eq:mutwo} and $\veps=10^{-8}$.
To find a minimizer $u_h^{(m)}$ of $\cJ$ in $\vphi_0^{(m)}+\Psi^{(m)}$ in Step \ref{algo:STEP-1} of the ASI Algorithm at the $m$-th ASI iteration, we use the BFGS quasi-Newton method \cite{NW2006}.
We stop the iterations of the BFGS method once the relative norm of the gradient of the function
\[
	\beta \longmapsto \cJ\Big[\vphi_0^{(m)} +\sum_{j=1}^{\PsiDim_m} \beta_j\vphi_j^{(m)} \Big]
\]
is smaller than $10^{-6}$.
The criterion for incrementing the  frequency is given by (\ref{eq:frequency-stepping-criterion}) with a tolerance $\varepsilon_{\nu}=0.005$.
In Step \ref{algo:STEP-4} of the ASI Algorithm, we use (\ref{eq:truncation-criterion}) with $\varepsilon_\Psi=\varepsilon_{\nu}=0.005$.

\subsection{Five simple geometric inclusions}

\begin{figure}[t]
\centering
    \begin{minipage}[t]{0.3\linewidth}
    \subcaptionbox{The medium $u$\label{fig:Exp1-ooo-adaptive-eigenspace-medium}}{%
                \includegraphics[width=\linewidth]{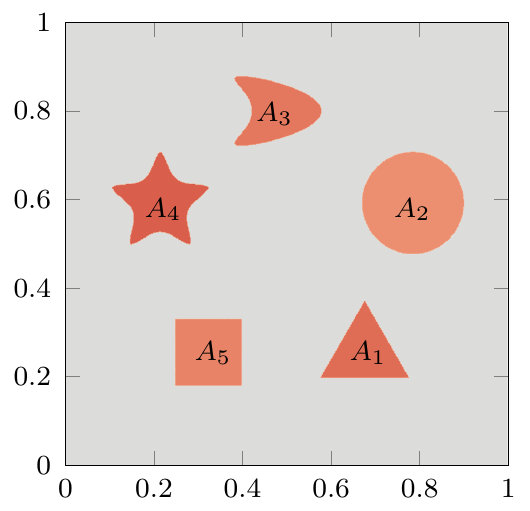}}
    \end{minipage}%
    \hfill
    \begin{minipage}[t]{0.3\linewidth}
    \subcaptionbox{Locations of the smoothed Gaussian point sources\label{fig:Exp1-ooo-adaptive-eigenspace-source}}{%
                \includegraphics[width=\linewidth]{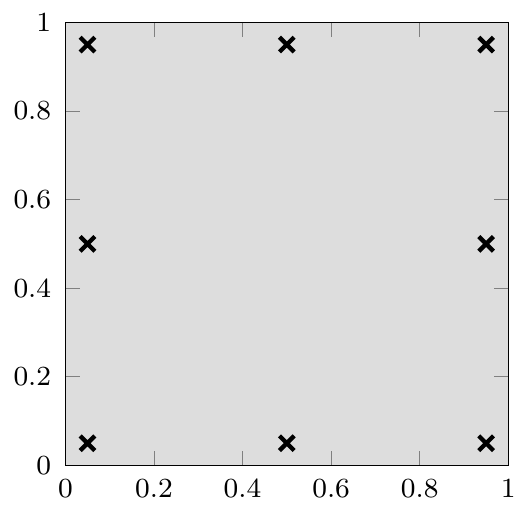}}
    \end{minipage}%
    \hfill
    \begin{minipage}[t]{0.3\linewidth}
    \subcaptionbox{The reconstruction $u_{h}^\textsuperscript{\smash{(30)}}$ at the 30th ASI iteration (rel.\ $L^2$ error of $4.4\%$)
        \label{fig:Exp1-ooo-adaptive-eigenspace-reconstruction}}{%
                 \includegraphics[width=\linewidth]{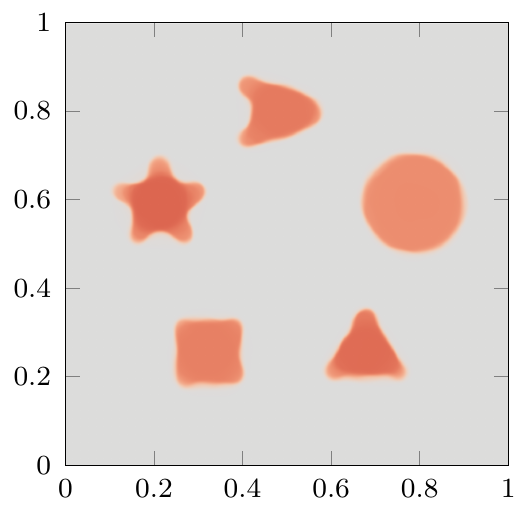}}
    \end{minipage}
    \caption{Five simple geometric inclusions:
    ASI reconstruction $u_h$ of the unknown medium $u$ from boundary observations with 20\% noise}
    \label{fig:Exp1-ooo-adaptive-eigenspace_1}
\end{figure}

\begin{figure}[t]
    \centering    
    \subcaptionbox{misfit $\cJ$ vs.\ ASI iteration \label{fig:Exp1-ooo-misfit}}{%
            \includegraphics[width=0.47\linewidth]{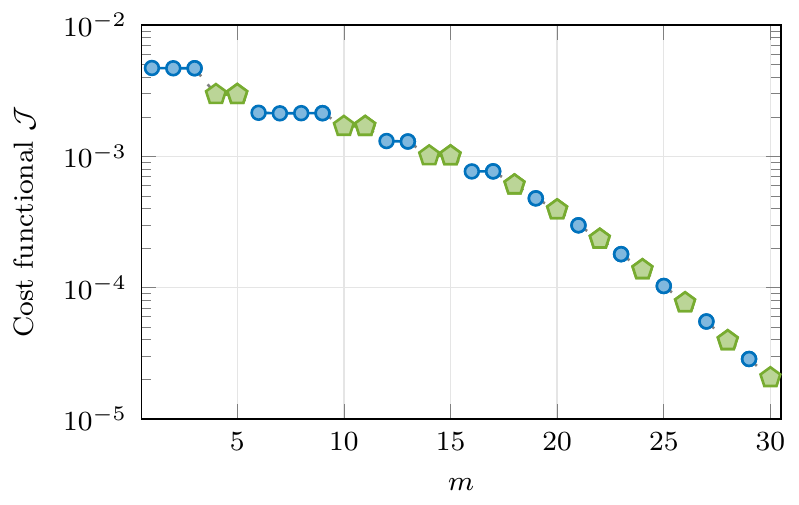}
    }
\hfill
    \subcaptionbox{relative $L^2$ error vs.\ ASI iteration \label{fig:Exp1-ooo-L2error}}{%
            \includegraphics[width=0.47\linewidth]{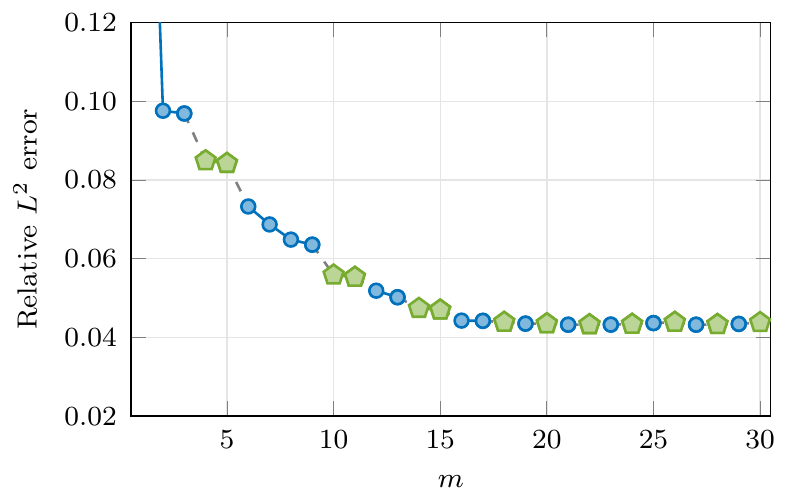}
    }
    \\[2ex]
    \subcaptionbox{frequency vs.\ ASI iteration \label{fig:Exp1-ooo-frequency}}{%
            \includegraphics[width=0.47\linewidth]{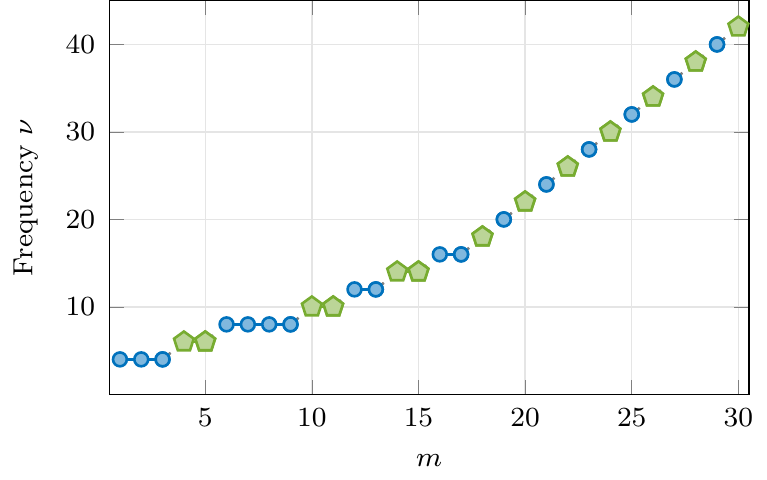}%
    }
\hfill
    \subcaptionbox{number of basis functions vs.\ ASI iteration \label{fig:Exp1-ooo-nb-basisfunction}}{%
            \includegraphics[width=0.47\linewidth]{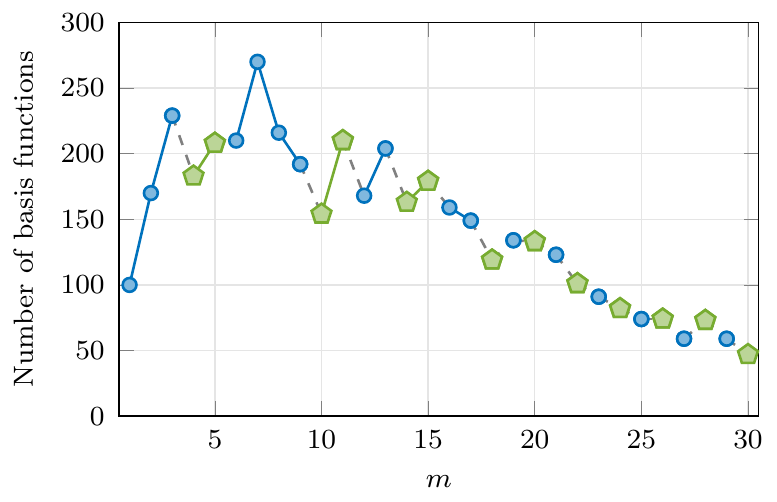}%
    }

    \caption{Five simple geometric inclusions.
    A change in marker indicates a change in frequency.
    }
\end{figure}

\begin{figure}[t]
    \centering
    \def\expFSlambdaA{16.92}
    \def\expFSlambdaB{22.82}
    \def\expFSlambdaC{22.88}
    \def\expFSlambdaD{25.18}
    \def\expFSlambdaE{25.59}
    \def\expFSlambdaF{159.98}
    \subcaptionbox{
    $\psi_1$ \quad ($\lambda_1 \approx \expFSlambdaA$)\label{fig:Exp1-ooo-adaptive-eigenspace-basisfunction-1}}{%
	\includegraphics[width=0.3\linewidth]{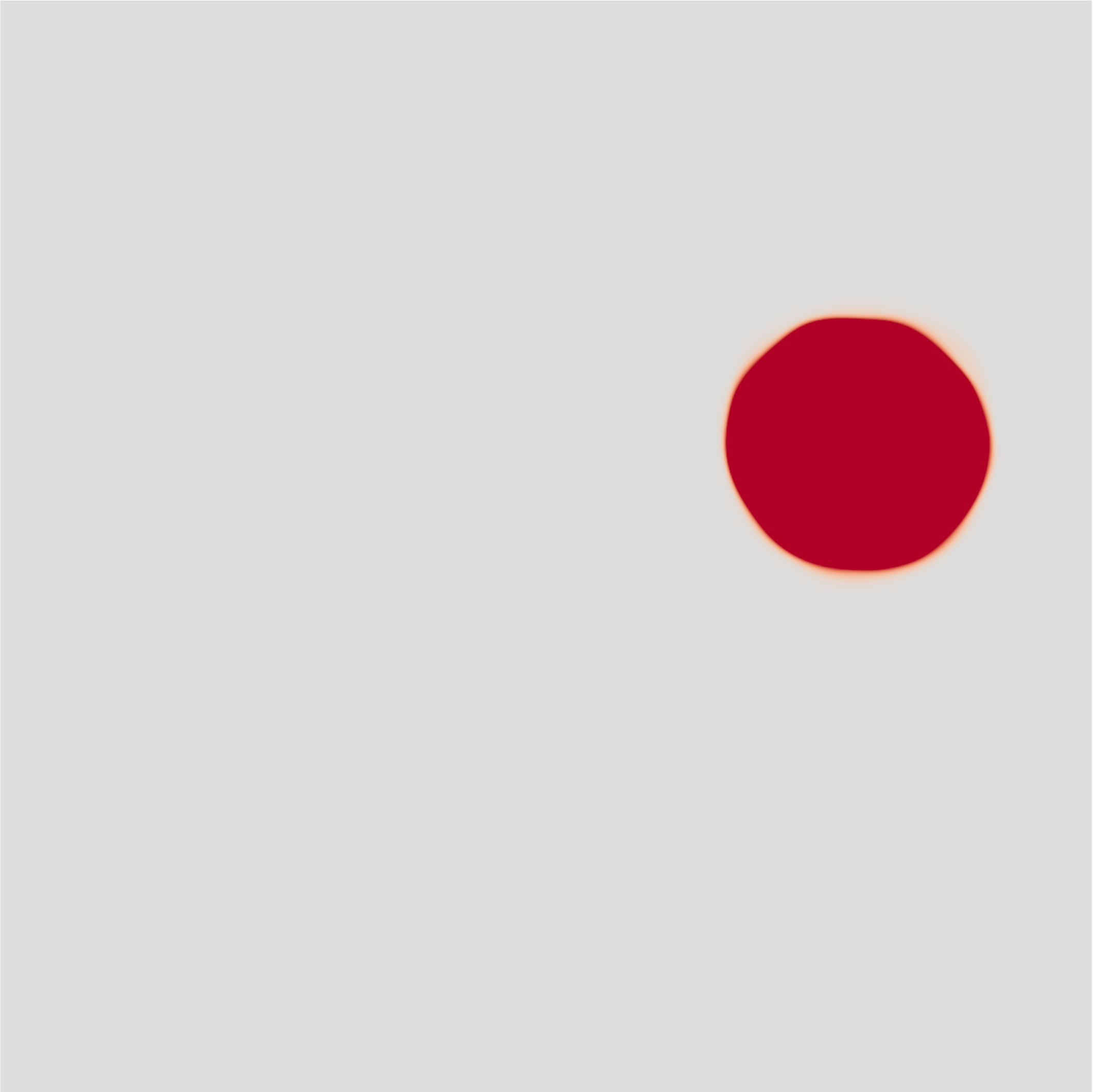}
        }
    \hspace{0.25cm}
    \subcaptionbox{
    $\psi_2$ \quad ($\lambda_2 \approx \expFSlambdaB$)\label{fig:Exp1-ooo-adaptive-eigenspace-basisfunction-2}}{%
	\includegraphics[width=0.3\linewidth]{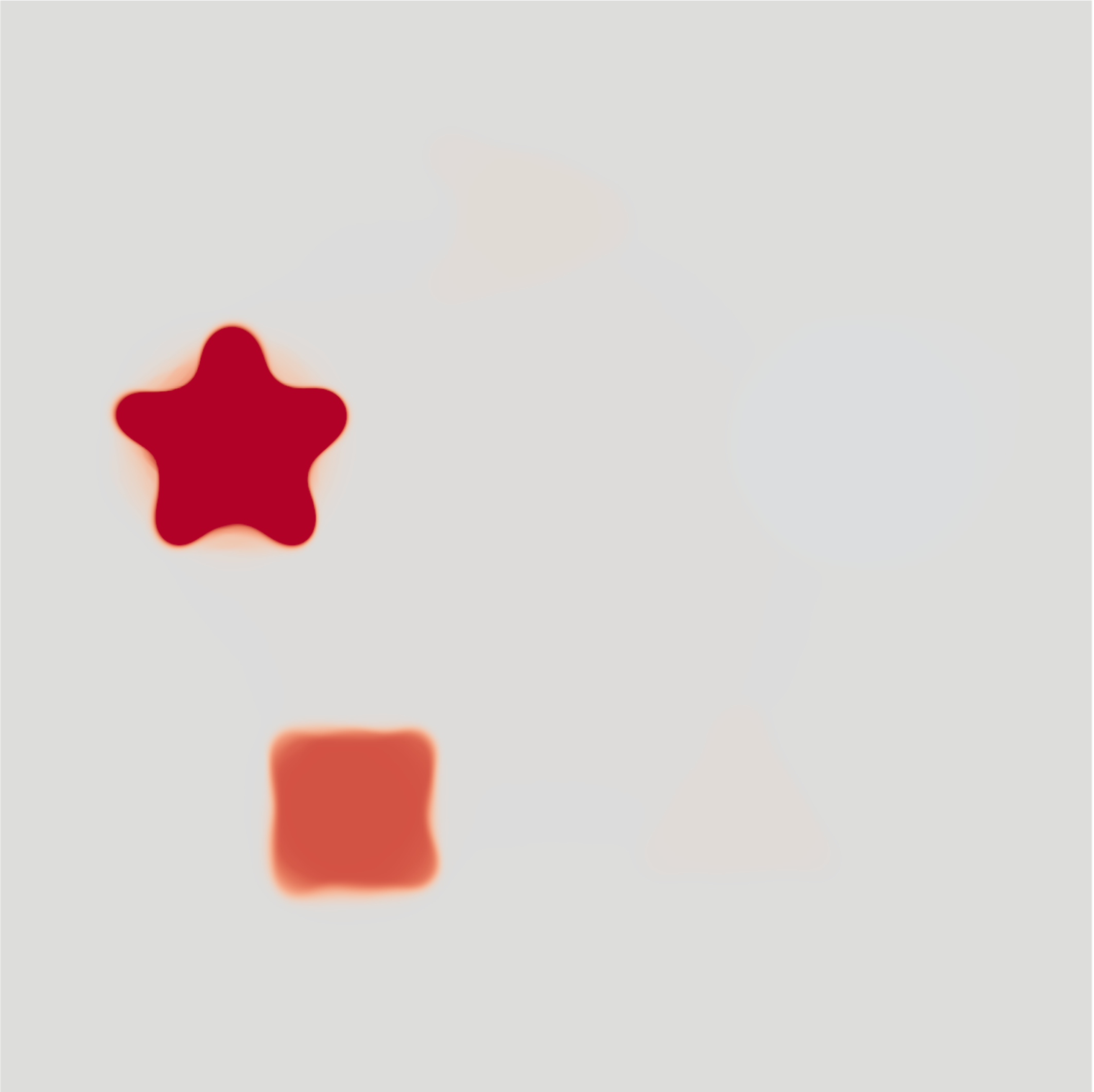}
        }
    \hspace{0.25cm}
    \subcaptionbox{
    $\psi_3$ \quad ($\lambda_3 \approx \expFSlambdaC$)\label{fig:Exp1-ooo-adaptive-eigenspace-basisfunction-3}}{%
	\includegraphics[width=0.3\linewidth]{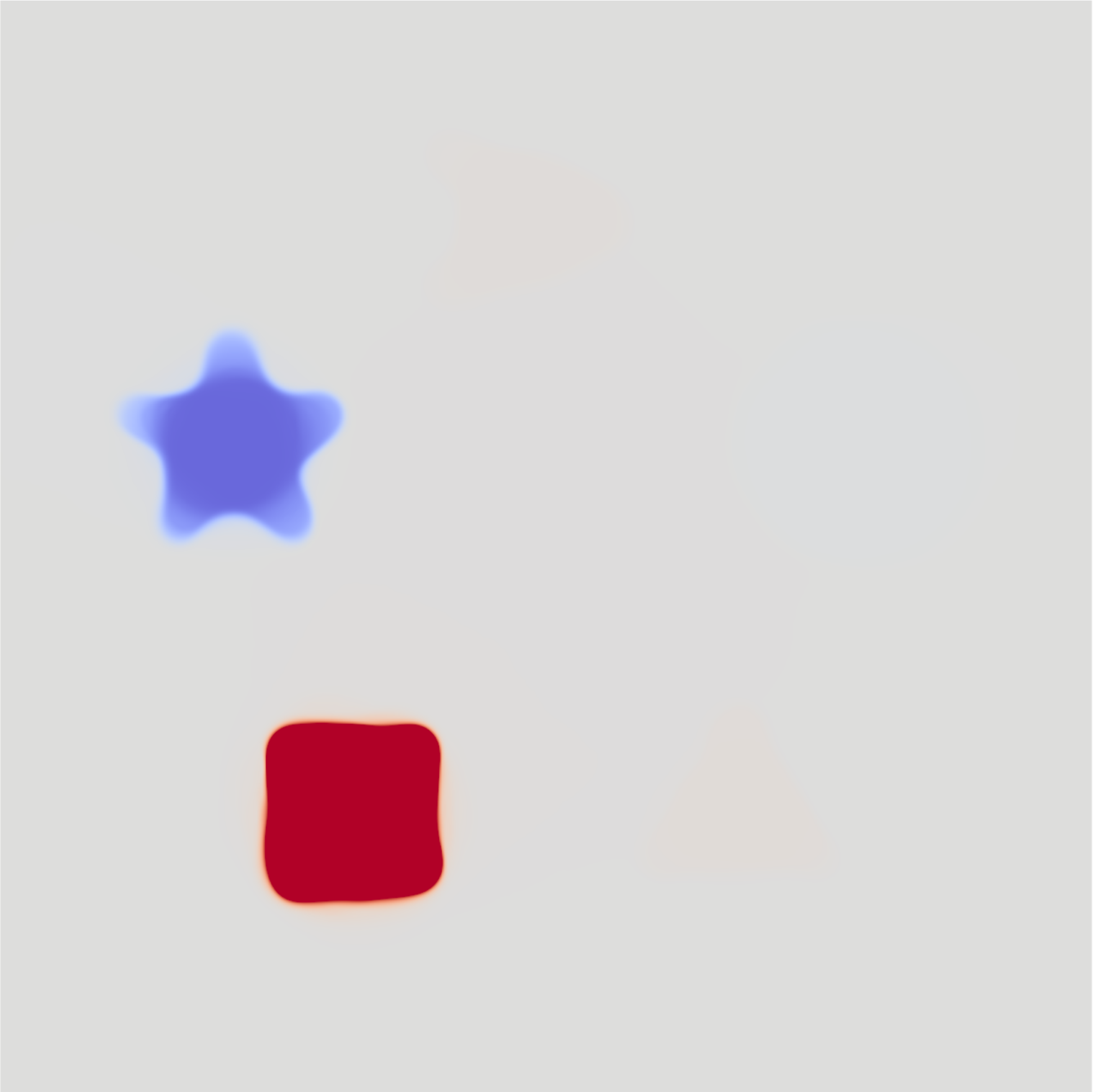}
        }
    \\[2ex]

    \subcaptionbox{
    $\psi_4$ \quad ($\lambda_4 \approx \expFSlambdaD$)\label{fig:Exp1-ooo-adaptive-eigenspace-basisfunction-4}}{%
	\includegraphics[width=0.3\linewidth]{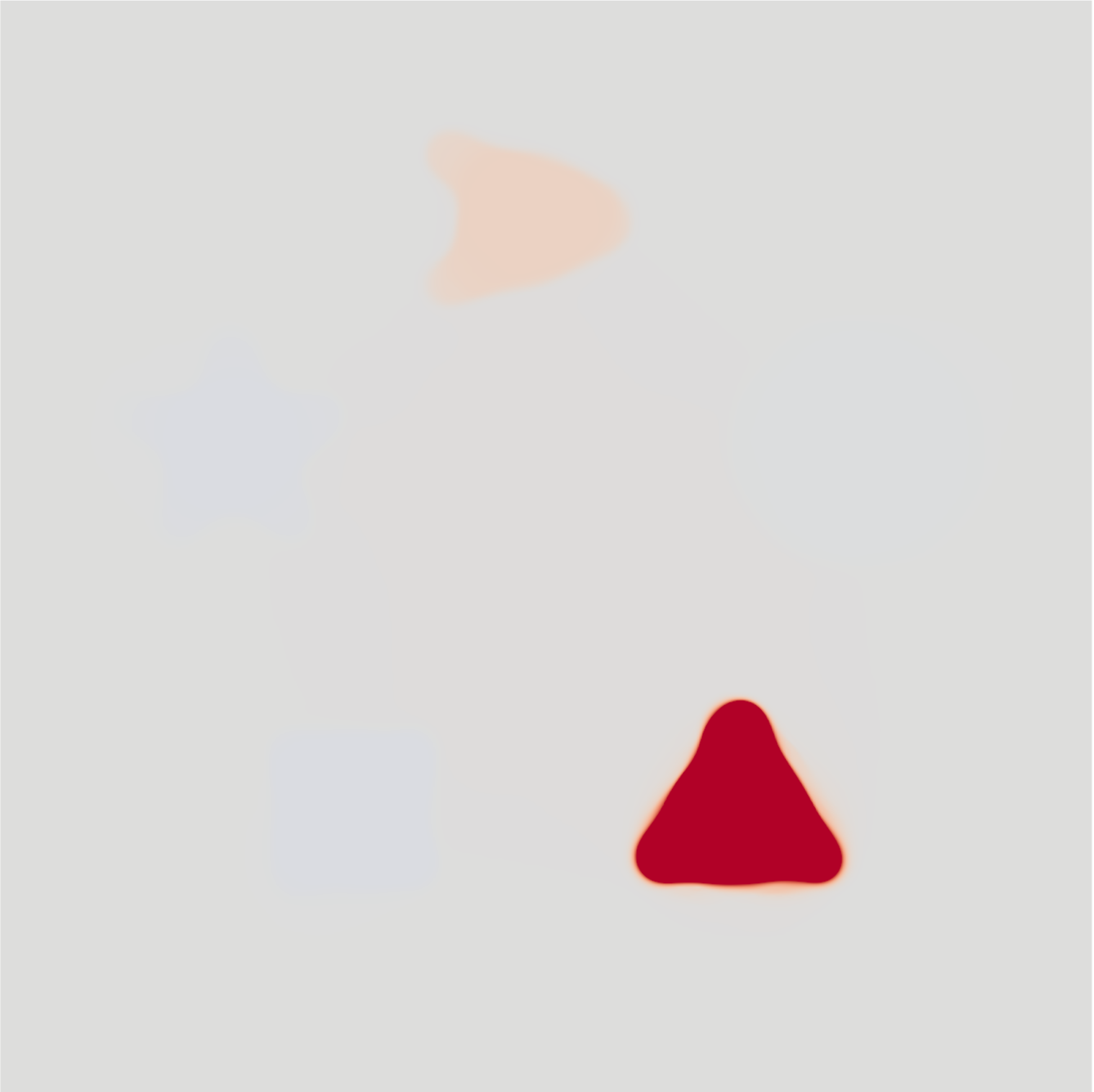}
        }
    \hspace{0.25cm}
    \subcaptionbox{
    $\psi_5$ \quad ($\lambda_5 \approx \expFSlambdaE$)\label{fig:Exp1-ooo-adaptive-eigenspace-basisfunction-5}}{%
	\includegraphics[width=0.3\linewidth]{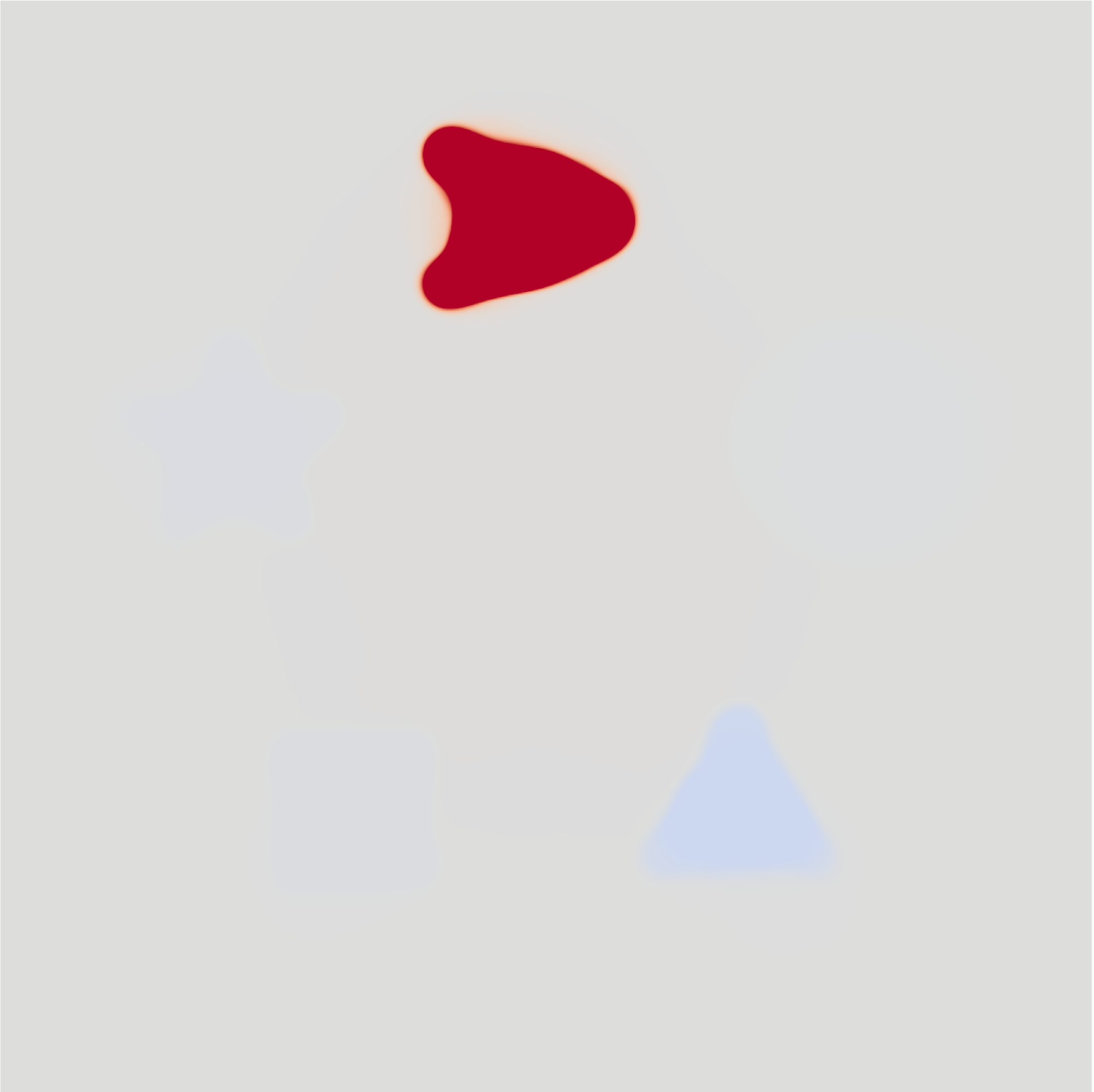}
        }

    \caption{Five simple geometric inclusions:
    the first five basis functions $\psi_1,\ldots,\psi_5$ and their corresponding
    Rayleigh quotients $\lambda_1,\ldots,\lambda_5$,
    at the final iteration of the ASI method.
    }
    \label{fig:Exp1-ooo-adaptive-eigenspace_2}
\end{figure}

We consider inside $\Omega=(0,1) \times(0,1)$ the unknown piecewise constant medium
\[
    u(x) = 2 + 1.4 \chi_{A_1}(x) + 1.1\chi_{A_2}(x) + 1.3\chi_{A_3}(x)+1.5\chi_{A_4}(x)+1.2\chi_{A_5}(x),
\]
where $\chi_{A_i}$ denotes the characteristic function of the set $A_i$, $i=1,\ldots,5$, shown in Figure \ref{fig:Exp1-ooo-adaptive-eigenspace-medium}.
Given noisy observations $\yObs_\ell$, $\ell=1,\ldots,N_S$, on the boundary $\G=\p\Om$ of $\Om$, we seek to reconstruct $u$ inside $\Om$.

To generate the perturbed observations $\yObs_\ell$, $\ell=1,\ldots,N_S$, we add $20 \%$ white noise to the numerical solutions of \deq{eq:HE} with $g=0$ for eight separate smoothed Gaussian point sources $f$, each centered at a location shown in Figure~\ref{fig:Exp1-ooo-adaptive-eigenspace-source}.
Both the observations $\yObs_\ell$ and the scattered wave fields $\y_\ell=y_\ell[w]$, needed for the inversion procedure, are discretized with a $\mathcal{P}^3$-FEM using about $10$ grid points per wavelength.
To avoid any inverse crime, the observations $\yObs_\ell$ are computed on a different, about $30\%$ finer mesh.

We seek our approximation $u_h$ of the exact medium $u$ in the $\mathcal{P}^1$-FE space for a regular, triangular mesh with $320'801$ vertices and $640'000$ elements.
Since $u$ is  assumed known on the boundary $\p\Om$ of $\Om$, where it is constant, we choose as initial guess $u_h^{(0)}$ constant throughout $\Om$; thus, on $\p\Om$, $u_h^{(0)}$ coincides with $u$.
To initialize the algorithm we must also choose $\varphi_0^{(1)}$ and $\Psi^{(1)}=\Span\{\psi_1^{(1)},\ldots,\psi_{J_1}^{(1)}\}$.
To do so, we set $J_1=100$ and solve \deq{eq:phi0BVP_alg_m} and \deq{eq:eigenValProb_alg_m}.
%
%
%
Since $u_h^{(0)}$ is constant, $L_\veps[u_h^{(0)}]$ simply reduces to the negative Laplacian times~$\veps^{-1}$.
For Step \ref{algo:STEP-5} of the ASI Algorithm, we use (\ref{eq:truncation-criterion-lowerbound}) with $\rho_0=0.8$ and $\rho_1=1.2$.


The reconstruction $u_h=u_h^{(m)}$ of the medium at the final ASI iteration, with $m=30$, is shown in Figure~\ref{fig:Exp1-ooo-adaptive-eigenspace-reconstruction}.
Although corners appear smoothed out in the reconstruction, $u_h$ captures well the locations and shapes of the inclusions composing $u$ with a relative $L^2$-error $\|u-u_h\|/\|u\|$ of $4.4\%$.

In Figures \ref{fig:Exp1-ooo-misfit} and \ref{fig:Exp1-ooo-L2error}, we show the misfit $\cJ$ of the reconstructed medium $u_h^{(m)}$ and the relative $L^2$ error at each ASI iteration.
Different markers indicate different frequencies, shown in Figure \ref{fig:Exp1-ooo-frequency}.
Both the misfit and the relative $L^2$ error decrease until the relative error eventually levels around 0.04.
In Figure \ref{fig:Exp1-ooo-nb-basisfunction}, we monitor the dimension $\PsiDim_m$ of the adaptive space $\Psi^{(m)}$ at each ASI iteration.
Starting at $\PsiDim_1=100$, first the dimension $\PsiDim_m$ increases until it peaks around $\PsiDim_m=270$ at iteration $m=7$, and eventually decreases to about 50.

The first five basis functions $\psi_1^{(m)},\ldots,\psi_5^{(m)}$ of the adaptive space $\Psi^{(m)}$ at the final step, $m=30$, shown in Figure \ref{fig:Exp1-ooo-adaptive-eigenspace_2}, illustrate how the adaptive basis captures the span of the first five lowest eigenfunctions in the AS decomposition.
\subsection{Salt-dome model}\label{subsec:SaltDome}
\begin{figure}
    \centering
    \subcaptionbox{unknown profile $u_\delta$\label{fig:Exp2-Pluto-profile}}{%
        \includegraphics[width=0.47\linewidth]{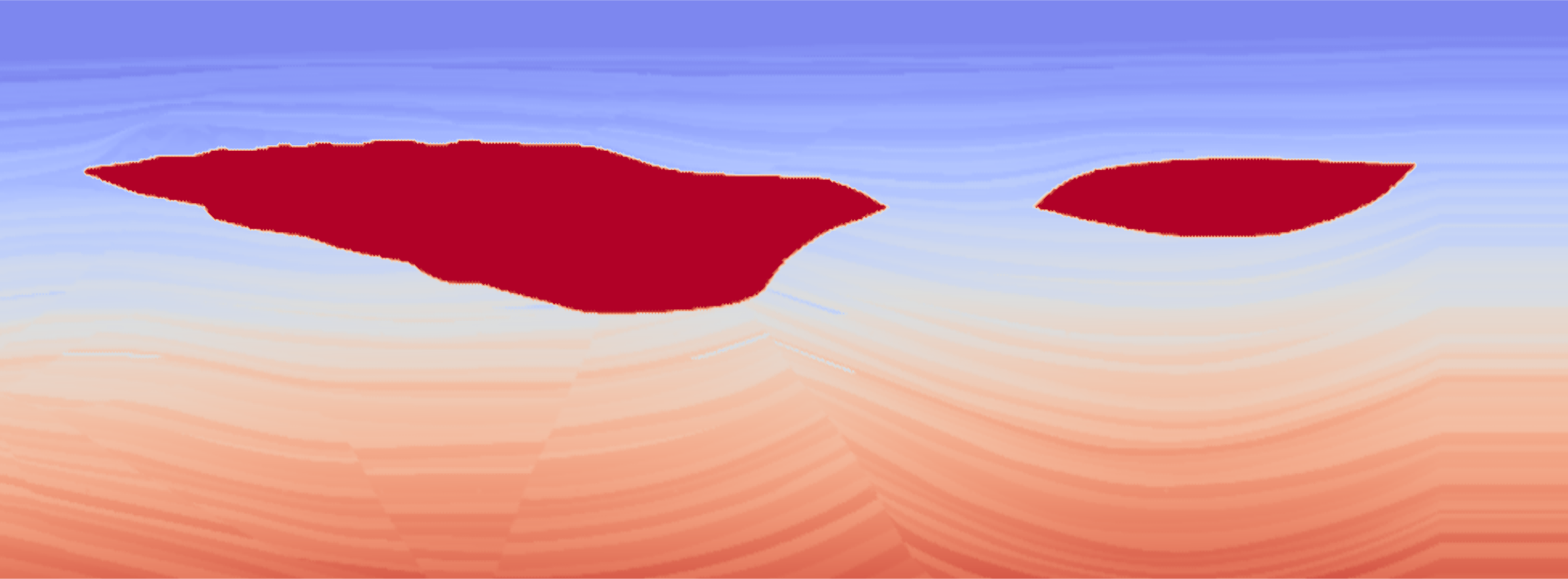}
    }%
\hfill
    \subcaptionbox{initial guess ($L^2$-error of $23.7\%$)\label{fig:Exp2-Pluto-initialguess}}{%
        \includegraphics[width=0.47\linewidth]{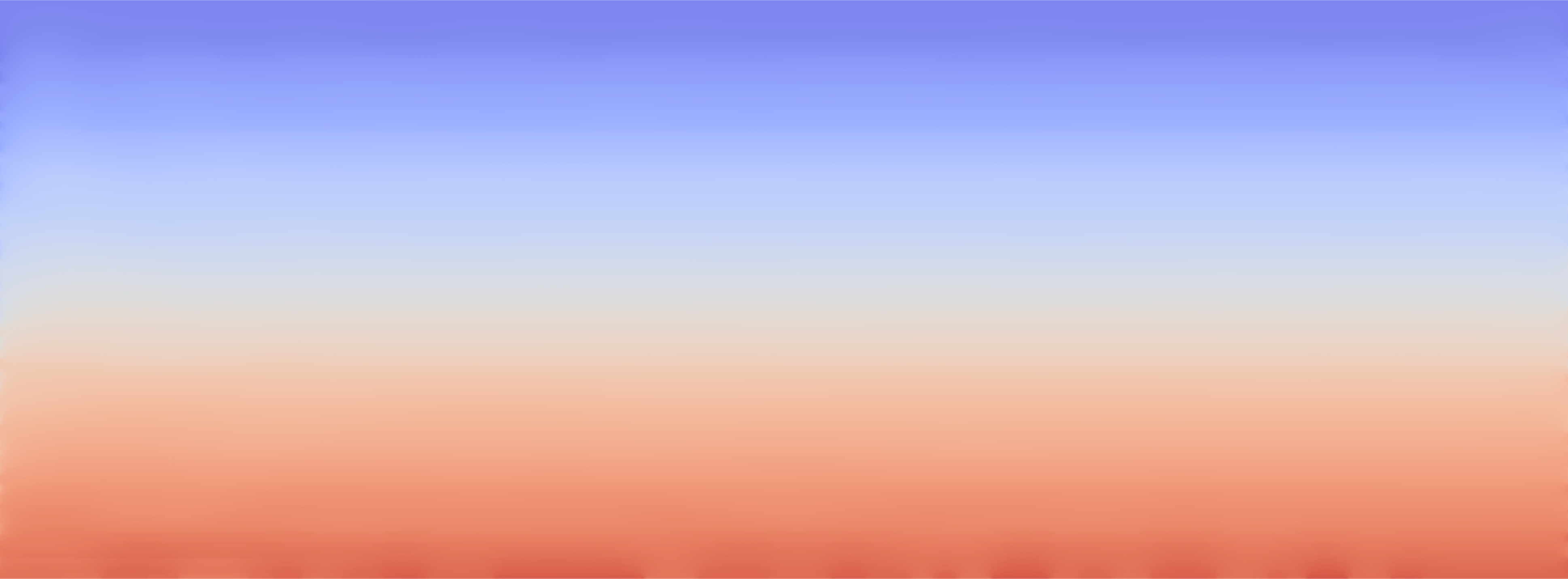}
    }\\[1em]
    \subcaptionbox{ASI reconstruction ($L^2$ error of $5.0\%$)}{%
    \includegraphics[width=0.47\linewidth]{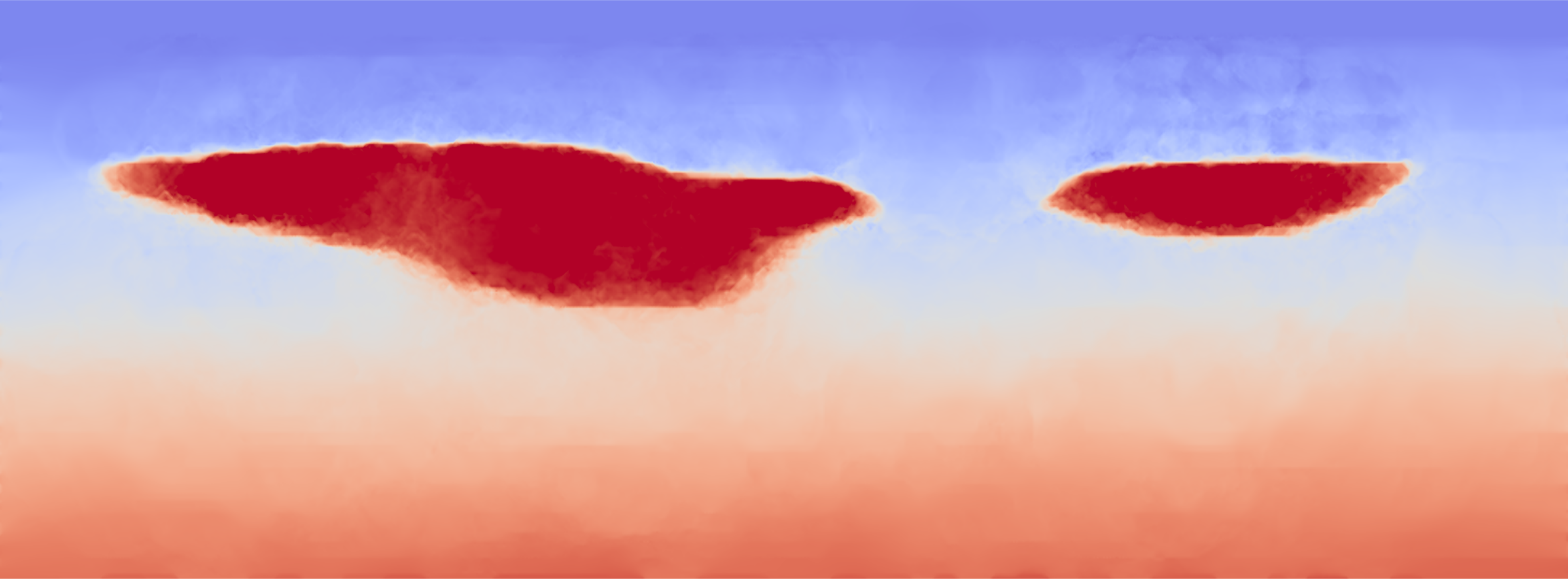}
    }

    \caption{Salt dome model:
    top left:  FE approximation $u_\delta$ of $u$;
    top right: initial guess $u_h^{(0)}$;
    bottom: ASI reconstruction ($\nu=4.0 [Hz]$).
    }
    \label{fig:Exp2-Pluto}
\end{figure}

\begin{figure}[t]
    \centering
    \subcaptionbox{misfit $\mathcal{J}(u_h^{(m)})$ vs.\ ASI iteration\label{fig:Exp2-Pluto-misfit}}{%
\includegraphics[width=0.47\linewidth]{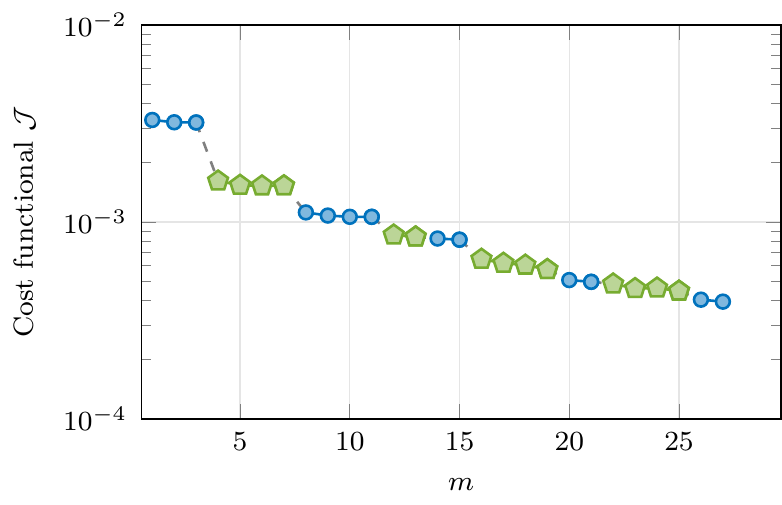}
        }
\hfill
    \subcaptionbox{relative $L^2$ error in $u_h^{(m)}$ vs.\ ASI iteration\label{fig:Exp2-Pluto-L2error}}{%
\includegraphics[width=0.47\linewidth]{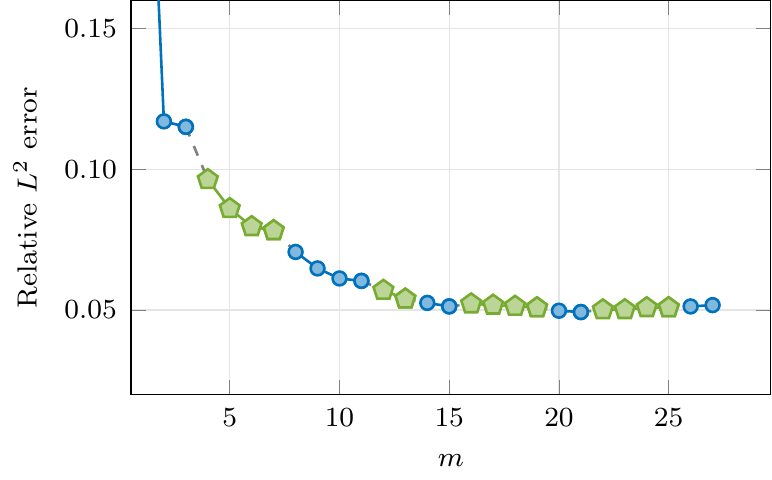}
        }
    \\[2ex]
    \subcaptionbox{frequency $\nu$ vs.\ ASI iteration\label{fig:Exp2-Pluto-frequency}}{%
\includegraphics[width=0.47\linewidth]{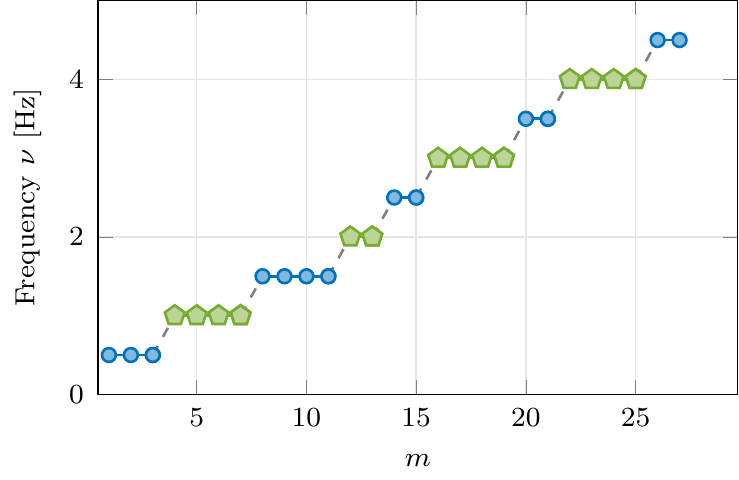}
        }
\hfill
    \subcaptionbox{number of basis functions $\PsiDim_m$ vs.\ ASI iteration\label{fig:Exp2-Pluto-nbfunctions}}{%
\includegraphics[width=0.47\linewidth]{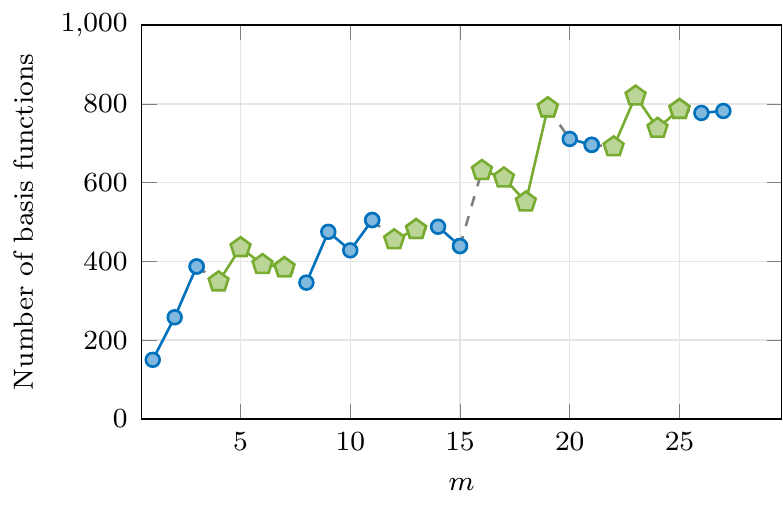}
        }
    \\[2ex]

    \caption{Salt dome model: misfit (a), relative $L^2$ error (b),
    frequency $\nu$ (c), and number of basis functions (d) in each ASI iteration.
    A change in marker indicates a change in frequency $\nu$.
    }
    \label{fig:Exp2-measurements}
\end{figure}

We consider a two-dimensional (Pluto 1.5) salt dome model from geosciences, generated by the ``subsalt multiples attenuation and reduction technology'' (SMAART)~\cite{B2002}.
Hence, we consider (\ref{eq:HE_PDE}) in $\Omega=(0,24.4) \times (-9,0)$ [km] with the squared velocity profile $u(x)=c^2(x)$ [km/s]$^2$ shown in Figure \ref{fig:Exp2-Pluto} (top left).
We impose a first order absorbing boundary condition (\ref{eq:HE_ABC}) on the two lateral and the lower artificial boundaries, and a homogeneous Neumann condition at the top (physical) boundary $\Gamma=\{y=0\}\subset \partial\Omega$.
Synthetic observations at the surface $\Gamma$ are obtained from the wave fields induced by $N_s=100$ source terms located $50$ meters beneath the top surface, each about $245$ meters apart.
Again, we add $20\%$ white noise to the observations.

Now, we apply the ASI Algorithm to reconstruct $u$ from the surface observations.
Here, we seek an approximation $u_h$ of the unknown medium $u$ in a $\mathcal{P}^1$-FE space with $176'349$ vertices and $351'360$ elements.
The initial guess $u^{(0)}$, shown in Figure \ref{fig:Exp2-Pluto}, is generated by extending along the $x$ direction the known Eastern boundary (borehole) data $u$ to the entire computational domain.
Then, starting with $\nu=0.5$ [Hz], the initial guess and the space $\vphi^{(1)}+\Psi^{(1)}$ spanned by the spectral basis of the negative Laplacian operator with $\PsiDim_1=\dim(\Psi^{(1)})=150$, we solve the inverse problem (\ref{eq:optimization-problem}) by using the ASI method while progressively increasing the frequency $\nu$ up to $4.0$ [Hz].
In Step \ref{algo:STEP-5} of the ASI Algorithm, we set $\rho_0=0.9$ and $\rho_1=1.1$ in (\ref{eq:truncation-criterion-lowerbound}).

After $m=27$ ASI iterations, the reconstruction $u_h=u_h^{(m)}$, shown in Fig.\ \ref{fig:Exp2-Pluto}, captures remarkably well the location, size, and inner velocities of the two salt bodies, though not originally present in the initial guess.
In Fig.\ \ref{fig:Exp2-measurements}, we monitor the misfit $\cJ(u_h^{(m)})$,  the dimension $\PsiDim_m$ of the search space $\Psi^{(m)}$, and the relative $L^2$-error at each ASI iteration.
Both the relative $L^2$-error and the misfit $\cJ$ monotonically decrease until the error levels off at about $5\%$.
In Fig.\ \ref{fig:Exp2-Pluto-nbfunctions}, we observe that the number of basis functions of the AS space ${\Psi}^{(m)}$ varies between $150$ and $800$ basis functions.

%
%
%
%
%
%
%
%
%
%
%
%
%
%
%
%
\section{Concluding remarks}
\label{sec:conc}
Starting from the adaptive spectral (AS) decomposition \eqref{eq:wExpansion}--\eqref{eq:eigenValProb}, we have proposed 
a nonlinear optimization method for the solution of inverse medium problems.
Instead of a grid-based discrete representation, the unknown medium $u(x)$
is projected at the $m$-th iteration to a finite-dimensional affine subspace $\vphi_0 + \Psi^{(m)}$, which is updated repeatedly.
The search space $\Psi^{(m)}$ is constructed by combining 
the search space from the previous iteration with 
the "background" $\vphi_0$, which satisfies
\eqref{eq:phi0BVP}, and the first $K$ eigenfunctions $\vphi_j$ of a judicious linear elliptic operator $L_\veps$. Since $L_\veps$ depends itself on the current iterate, $\vphi_0$ and the orthonormal basis of
eigenfunctions $\vphi_1, \dots, \vphi_K$ are updated repeatedly.
Moreover, the resulting ASI (adaptive spectral inversion) algorithm, listed in Section 2.2, 
also adapts "on the fly" the dimension $J_m$ of the search space
by solving a small, quadratically constrained, quadratic minimization problem to filter basis functions while preserving important features. Hence the AS decomposition not only 
substantially reduces the dimension of the search space, but also removes the need for added Tikhonov-type 
regularization. Our numerical results for the ASI method, when applied to time-harmonic inverse scattering
problems governed by the Helmholtz equation, illustrate its accuracy and efficiency even in
the presence of noisy or partial boundary data.  In particular, the ASI method is able to invert
a two-dimensional (Pluto 1.5) salt dome model~\cite{B2002} from noisy surface observations with only a few hundred control variables. 

Our analysis in Section  \ref{sec:AE_analysis} underpins the remarkable accuracy of the AS decomposition 
observed in practice. In particular, Theorem \ref{thm:main} provides rigorous estimates for the approximation $\vphi_0$ of the background and for the eigenfunctions $\vphi_1,\ldots,\vphi_K$ of $L_\veps[u_\del]$, when the medium $u$
consists of $K$ piecewise constant distinct characteristic functions.
Our estimates imply that $\vphi_0$ and the first $K$ eigenfunctions $\vphi_j$ are ``almost'' constant in each connected component away from interfaces.
Hence for small $\del$ and $\veps$, the background is well approximated by $\vphi_0$, whereas
the deviation from the background, $u-\vphi_0$ (or $u_\del-\vphi_0$), may be well approximated in the span of $\vphi_1,\ldots,\vphi_K$. 
The analysis is valid for a wide class of medium-dependent (nonlinear) weight functions $\mu_\veps$ and also for different types of $H^1$-regular approximations $u_\del$ of $u$. In particular, 
they also hold for standard $H^1$-conforming FE approximations $u_\del = u_h$, where $\del$ corresponds
to the underlying mesh size $h$ -- see Corollary \ref{cor:main}.

The operator $L_\veps[u_\del]$ is related to the linearization of the 
total variation functional \cite{GKN2017}.
In Remark \ref{remark:MainThm}, we discuss similarities between Theorem \ref{thm:main} and the nonlinear spectral theory for the TV functional \cite{BCEGM2016,BGM2016,BCN2002}.
The comparison also raises some interesting questions we have not addressed in this work.
Since our estimates are uniform in the parameters $\veps$ and $\del$, they suggest
that a limit argument could be used for estimating the asymptotic behavior of the eigenvalues 
for $\veps, \del\rightarrow 0$.
Interestingly, however, the present estimates are also valid for cases that are specifically excluded by 
the spectral theory for the TV functional \cite{BCN2002}.
More precisely, according to \cite{BCN2002}, if a set $A$ is not convex, or if its  boundary is Lipschitz but not $C^{1,1}$, then its characteristic function $\chi_A$ 
cannot be an eigenfunction for the eigenvalue \deq{TV_EV} of the TV functional.
Our estimates, however, then still hold true, as supported by our numerical tests in Section \ref{sbSec:Numerical example}.

Although we have only considered scalar problems here, the AS decomposition 
can also be used for multiple parameters \cite{GN2019}. As the AS decomposition is independent
of the underlying governing PDE, or any particular choice of misfit functional $\cJ[u]$,  it is probably also useful for other inverse problems with time-dependent or elliptic governing field equations, or possibly for different applications from image analysis.

\bigbreak\bigbreak

\bibliographystyle{abbrv}
\bibliography{main-adaptive-eigenspace.bib,Refs.bib}

\end{document}